\documentclass[11pt]{amsart}

\usepackage{enumitem}
\usepackage[utf8]{inputenc}
\usepackage[T1]{fontenc}

\usepackage[english]{babel}
\usepackage{fullpage}
\usepackage{enumitem} 
\usepackage{bbm}
\usepackage{dsfont}

\usepackage{amssymb,libertine}
\usepackage{amssymb}
\usepackage{array}
\usepackage{bbm}

\usepackage{mathtools}
\usepackage[svgnames]{xcolor}
\usepackage{hyperref}
\usepackage{comment}
\usepackage[colorinlistoftodos]{todonotes}

\usepackage{mathrsfs}
\usepackage{wrapfig}
\usepackage{color}
\usepackage{tikz}

\usepackage{graphicx}
\usepackage{epstopdf}

 \usepackage{latexsym}
 \usepackage{amsfonts}

\usepackage{caption}
\usepackage{subcaption}

\usepackage{amsmath}
\usepackage{amsthm}
\usepackage[export]{adjustbox}

    \def\RR{\mathbb{R}}
    
    \def\Re{\mathfrak{R}}

\newtheorem{theorem}{Theorem}[section]
\newtheorem{corollary}{Corollary}[theorem]
\newtheorem{lemma}{Lemma}[section]

\theoremstyle{definition}

\newtheorem{Remark}{Remark}[section]
\newtheorem{Theorem}{Theorem}[section]

\def\be{\begin{equation}}
\def\ee{\end{equation}}
    
\def\ge{\geqslant}
\def\le{\leqslant}

\def\bd{\begin{Definition}}
\def\ed{\end{Definition}}
\def\bt{\begin{Theorem}}
\def\et{\end{Theorem}}

 \def\bt{\begin{Remark}}
\def\et{\end{Remark}}

\allowdisplaybreaks

\def\epsilon{\varepsilon}
\def\bel{\begin{equation}\label}
\def\ee{\end{equation}}

\def\Ei\text{Ei}

\def\phi{\varphi}

\usepackage{graphicx}
\title{Long time evolution of the H\'enon-Heiles system for small energy}
\date{}
\author{Ovidiu Costin\textsuperscript{1}, Rodica Costin\textsuperscript{2}, Kriti Sehgal\textsuperscript{3}}
\thanks{\textsuperscript{1}The Ohio State University. Email: \href{mailto:costin.9@osu.edu}{costin.9@osu.edu}}
\thanks{\textsuperscript{2}The Ohio State University. Email: \href{mailto:costin.10@osu.edu}{costin.10@osu.edu}}
\thanks{\textsuperscript{3}The University of Chicago. Email: \href{mailto:ksehgal@uchicago.edu}{ksehgal@uchicago.edu}}

\begin{document}

\maketitle
\section{Abstract}
The H\'enon-Heiles system, initially introduced as a simplified model of galactic dynamics, has become a paradigmatic example in the study of nonlinear systems. Despite its simplicity, it exhibits remarkably rich dynamical behavior, including the interplay between regular and chaotic orbital dynamics, resonances, and stochastic regions in phase space, which have inspired extensive research in nonlinear dynamics.

In this work, we investigate the system's solutions at small energy levels, deriving asymptotic constants of motion that remain valid over remarkably long timescales—far exceeding the range of validity of conventional perturbation techniques. Our approach leverages the system's inherent two-scale dynamics, employing a novel analytical framework to uncover these long-lived invariants.

The derived formulas exhibit excellent agreement with numerical simulations, providing a deeper understanding of the system's long-term behavior.


\tableofcontents

\newpage
\section{Introduction}

In 1964, Michel Hénon and Carl Heiles introduced a model for the planar motion of a star under the influence of a galactic center (represented by a rotationally symmetric potential). This model, motivated by the question of the existence of a third integral for galactic motion, is both analytically simple to formulate and "sufficiently complicated to give trajectories that are far from trivial" \cite{HenonHeiles}. Hénon and Heiles’ numerical investigations revealed highly intricate trajectory behavior, including an infinite number of islands where some trajectories remain confined, chains linking these islands, and ergodic trajectories densely filling the surrounding region.

The Hénon-Heiles potential is given by \begin{equation} \label{potential} V(x,y) = \dfrac{1}{2}(x^2 + y^2) + x^2y - \dfrac{1}{3}y^3, \end{equation} which can be interpreted as two harmonic oscillators coupled by a cubic "perturbation." The corresponding Hamiltonian is \begin{equation} \label{hamiltonian} h = \dfrac{1}{2} (\dot{x}^2 + \dot{y}^2) + V(x,y), \end{equation} and the equations of motion are \begin{equation} \begin{aligned} \frac{d^2x}{dt^2} &= -x - 2xy, \  \ \ \frac{d^2y}{dt^2} &= -y + y^2 - x^2. \end{aligned} \end{equation}

From a theoretical perspective, this is a four-dimensional system of ordinary differential equations with a resonant fixed point at the origin. The dynamics near such resonant fixed points remain poorly understood, presenting an enduring theoretical challenge.

The numerical results of Hénon and Heiles have inspired extensive mathematical investigations into this system, employing diverse analytical and numerical techniques. These include studies of fractal structures \cite{barrio2008fractal}, escape dynamics \cite{blesa2012escape}, separability conditions \cite{ravoson1993separability}, and chaotic transitions \cite{Ito}, as well as broader explorations of integrability properties \cite{Casti},  \cite{Chrch}, \cite{Conte}, \cite{fordy1991henon}, \cite{Zh} and numerical studies \cite{Zotos}. This rich body of work underscores the continued relevance of the Hénon-Heiles system as a benchmark for understanding nonlinear dynamics.

For energies $h$ with $0 < h < 1/6$, the trajectories are confined within a triangular region bounded by the equipotential curve $h = 1/6$. At higher energies, trajectories can escape to infinity \cite{Zotos}.

In this paper, we analyze the trajectories for small variables; these evolve within the triangular confinement region. We provide asymptotic formulas for the solutions, valid for time periods significantly longer than those for which simple perturbation series are valid.

\vspace{10pt}

\subsection{Features Revealed Numerically}

Numerical calculations reveal several interesting features of the trajectories. Figure \ref{fig:evolution}
 plots the curve $\left(x(t),y(t)\right)$ for a numerically obtained solution of the Hénon-Heiles system with initial conditions $x(0) = 0.1$, $y(0) = 0$, $\dot{x}(0) = 0.08$, and $\dot{y}(0) = 0.1$ (for which $h = 0.0132 < 1/6$). The following features are observed: (A) over a short time, the trajectory is almost periodic; (B) each nearly closed trajectory drifts over time; (C) over a long time, the trajectory densely fills a domain.

\begin{figure}[h!] \centering \begin{subfigure}[b]{0.32\textwidth} \centering \includegraphics[width=\textwidth]{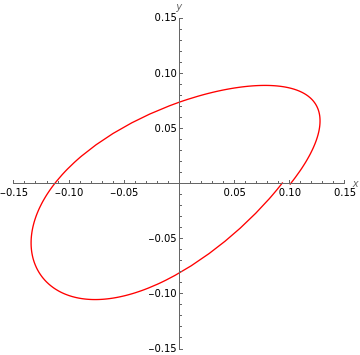} \caption{For $t = 6.27$} \label{fig:loop}\end{subfigure} \hfill \begin{subfigure}[b]{0.32\textwidth} \centering \includegraphics[width=\textwidth]{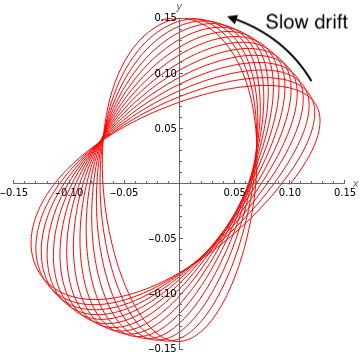} \caption{For $t = 90$} \label{fig:drift}\end{subfigure} \hfill \begin{subfigure}[b]{0.32\textwidth} \centering  \includegraphics[width=\textwidth]{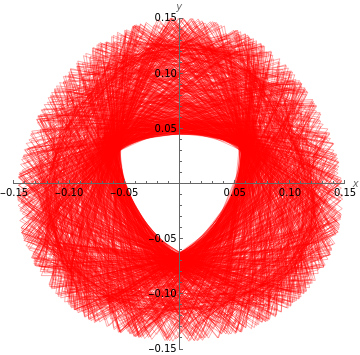} \caption{For $t = 4000$} \label{fig:long_time}\end{subfigure} \captionsetup{width = \linewidth} \caption[Numerically calculated evolution of the system]{Numerically calculated evolution of the system for initial conditions $x(0) = 0.1$, $y(0) = 0$, $\dot{x}(0) = 0.08$, and $\dot{y}(0) = 0.1$.} \label{fig:evolution} \end{figure}

We prove that the behaviors illustrated in (A) and (B) are mathematically correct. We derive explicit approximations for the solutions and provide rigorous estimates for the time span over which these approximations remain valid.
\subsection{Rescaling}

To investigate the system in the regime of small variable values, we introduce a small parameter \(\epsilon\) through a rescaling of the variables. Specifically, we apply the transformations \(x \mapsto \epsilon x\) and \(y \mapsto \epsilon y\). Under this rescaling, the Hénon-Heiles system assumes a ``perturbed'' form:

\begin{align}
\label{perturbedsystem}
\begin{split}
    \frac{d^2x}{dt^2} &= -x - 2 \epsilon x y, \\
    \frac{d^2y}{dt^2} &= -y + \epsilon y^2 - \epsilon x^2.
\end{split}
\end{align}
The Hamiltonian rescales as \(h \mapsto h / \epsilon^2\); for simplicity, we will continue to denote it by \(h\):
\begin{equation}
\label{hepsilon}
    h = \dfrac{1}{2} \left(x^2 + y^2 + \dot{x}^2 + \dot{y}^2 \right) + \epsilon x^2 y - \dfrac{\epsilon}{3} y^3.
\end{equation}

It is important to note that the speed of the slow evolution decreases as \(\epsilon\) decreases, as shown in Figure \ref{fig:epsilon_effect}.

\begin{figure}[hbt!]
     \centering
     \begin{subfigure}[b]{0.32\textwidth}
         \centering
         \includegraphics[width=\textwidth]{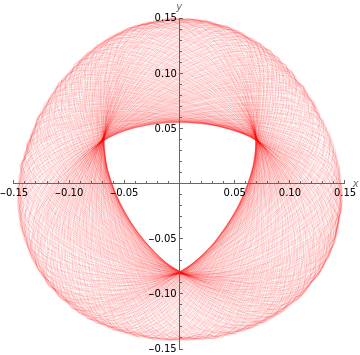}
         \caption{$\epsilon = 1$}
     \end{subfigure}
     \hfill
     \begin{subfigure}[b]{0.32\textwidth}
         \centering
         \includegraphics[width=\textwidth]{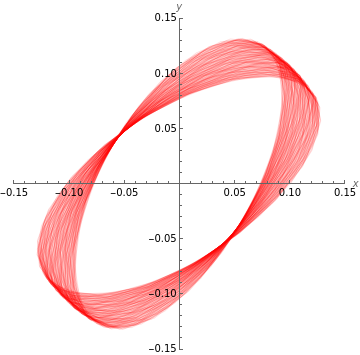}
         \caption{$\epsilon = 0.2$}
     \end{subfigure}
     \hfill
     \begin{subfigure}[b]{0.32\textwidth}
         \centering
         \includegraphics[width=\textwidth]{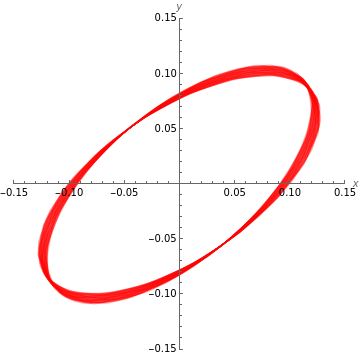}
         \caption{$\epsilon = 0.1$}
     \end{subfigure}
     \captionsetup{width = \linewidth}
     \caption{Comparison of the evolution of the system with initial conditions $x(0) = 0.1$, $y(0) = 0$, $\dot{x}(0) = 0.08$, and $\dot{y}(0) = 0.1$ from time $t = 0$ to $t = 1000$ for various values of \(\epsilon\).}
     \label{fig:epsilon_effect}
\end{figure}

\subsection{Perturbation series solution}

It is natural to attempt solving the system \eqref{perturbedsystem} using a perturbation series in \(\epsilon\). Introducing an expansion of the form
\begin{equation}
    \label{xyperturbseries}
    \begin{split}
        x(t) &= x_0 \cos (t) + \dot{x}_0 \sin (t) + \epsilon x_{\epsilon}(t) + \epsilon^2 x_{\epsilon^2}(t) + O(\epsilon^3), \\
        y(t) &= y_0 \cos (t) + \dot{y}_0 \sin (t) + \epsilon y_{\epsilon}(t) + \epsilon^2 y_{\epsilon^2}(t) + O(\epsilon^3),
    \end{split}
\end{equation}
a straightforward calculation yields the terms of this expansion explicitly, and these are given in Appendix \ref{epsilon square term}. We observe that \(x_{\epsilon}(t)\) and \(y_{\epsilon}(t)\) are \(2\pi\)-periodic, while the coefficients of \(\epsilon^2\) take the form:
\[
x_{\epsilon^2}(t) = f_1(t) + t f_2(t), \quad y_{\epsilon^2}(t) = g_1(t) + t g_2(t),
\]
where \(f_j(t)\) and \(g_j(t)\) (\(j = 1,2\)) are \(2\pi\)-periodic functions. This generates a {\em secular term}  \(t\epsilon^2\), which implies that the perturbation expansion holds only for times that are not too large, namely up to \(t\epsilon^2 = O(1)\).

We also observe the ``almost periodic'' nature of the motion. The time \(T\) required for a solution with initial conditions \(x_0, y_0, \dot{x}_0, \dot{y}_0\) to return to \(y(T) = y_0\) can be determined using the expansion above,  yielding
\begin{equation}
\label{time one loop}
T = 2\pi + \epsilon^2 \pi  \frac{14 x_0 y_0 \dot{x}_0 - 9 x_0^2 \dot{y}_0 + 5 y_0^2 \dot{y}_0 + 5 \dot{x}_0^2 \dot{y}_0 + 5 \dot{y}_0^3}{6 \dot{y}_0} + O(\epsilon^3), \quad \text{provided} \ \dot{y}_0 \neq 0.
\end{equation}
If \(\dot{y}_0 = 0\), the solutions do not exactly return to \(y_0\), but instead come close to it. In such cases, we find "approximate constants" valid for a long time, as described in Theorem \ref{mainth}(i).

As illustrated by the figures, the system exhibits multiscale behavior. This lies at the core of the secular terms—a well-known phenomenon \cite{Arnold}—which ultimately limits the validity of the expansion to a timescale too short to capture the system's intricate long-term dynamics.

To address this issue, a variety of multiscale methods have been introduced in both mathematics and physics \cite{Multiscale}. However, in this particular problem, the application of classical multiscale approaches proved to be unwieldy. Instead, we utilize "approximate adiabatic invariants", an approach and methods introduced in \cite{costin2016direct}, \cite{costin2015tronquee}.


\section{Main Results}

\subsection{Method used}

As in \cite{costin2016direct}, \cite{costin2015tronquee}, we use the Poincar\'e map to eliminate the fast variable, reducing the problems to a purely slow evolution one.

We define a variable \(u := \Phi(x, \dot{x}, y, \dot{y})\) as "slow" if \(\frac{d}{dt}\Phi(x, \dot{x}, y, \dot{y}) = O(\epsilon)\). It turns out that the homogeneous quadratic polynomials that represent slow variables are \(x^2 + \dot{x}^2\), \(y^2 + \dot{y}^2\), \(xy + \dot{x}\dot{y}\), and \(\dot{x}y - x\dot{y}\) (as well as their combinations).

We choose the following slow variables:
\begin{equation}\label{vw}
v = y^2 + \dot{y}^2, \quad w = \dot{x}\dot{y} + xy,
\end{equation}
in addition to \(h\) (which is not only slow but in fact constant). It turns out that our final result is more aesthetically pleasing when expressed in terms of \(u := h - v\) (another slow variable) instead of \(v\), and we will adopt this convention, though we will occasionally return to \(v\) during the proof.

Theorem\,\ref{mainth} summarizes our main results: we provide approximations for solutions that are valid for times much longer than those obtained from perturbation series expansions \eqref{xyperturbseries}.

Before stating the rigorous results, it is useful to observe the behavior of \(v(t)\) and \(w(t)\) numerically, and compare them with a non-slow variable. In figure \ref{fig:vwgraphs} (A) and (B), we observe that \(v(t)\) and \(w(t)\) exhibit a sinusoidal behavior. For \(\epsilon = 0.1\), the slow period for both \(v(t)\) and \(w(t)\) is approximately 7570. During this time, in a simple \(\epsilon\)-expansion, the secular terms grow to as large as 75.7, causing the expansion to fail after a small fraction of the period.
Superimposed over this slow sinusoidal behavior, we observe fast oscillations of small amplitude. In contrast, the fast variable \(x(t)\), shown in Figure \ref{fig:vwgraphs} (C), oscillates rapidly within a larger range. 
\footnote{The plot in Figure \ref{fig:vwgraphs} was obtained by numerically solving the system \eqref{perturbedsystem} with initial conditions \(x_0 = \frac{\sqrt{3/5}}{2}, y_0 = 0, \dot{x}_0 = \frac{1}{5}, \dot{y}_0 = \frac{1}{10}\), and \(\epsilon = 0.1\) (for which \(h = 0.1\), smaller than \(\tfrac{1}{6}\)). We used Mathematica's \texttt{NDSolve}, with \texttt{AccuracyGoal} set to 19. From this, we computed \(v(t)\) and \(w(t)\).}

\begin{figure}[hbt!]
     \centering
     \begin{subfigure}[b]{0.32\textwidth}
         \centering
         \includegraphics[width=\textwidth]{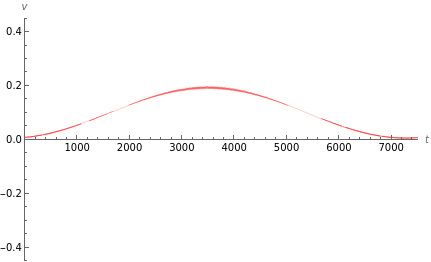}
         \caption{$v(t)=y(t)^2 + \dot{y}(t)^2$}
     \end{subfigure}
     \hfill
     \begin{subfigure}[b]{0.32\textwidth}
         \centering
         \includegraphics[width=\textwidth]{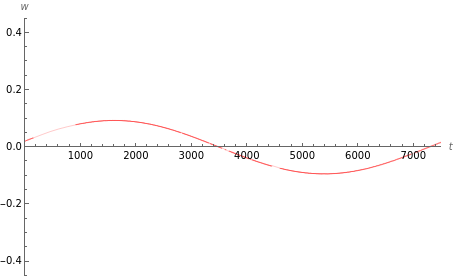}
         \caption{$w(t)=x(t)y(t) + \dot{x}(t)\dot{y}(t)$}
     \end{subfigure}
     \hfill
     \begin{subfigure}[b]{0.32\textwidth}
         \centering
         \includegraphics[width=\textwidth]{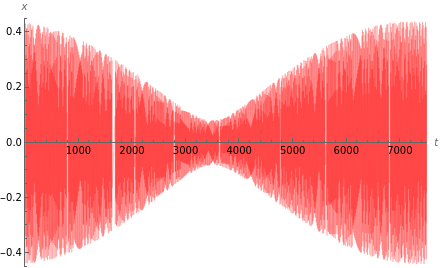}
         \caption{Fast variable $x(t)$}
     \end{subfigure}
     \captionsetup{width = \linewidth}
     \caption[Slow variables $v$ and $w$]{The evolution of the slow variables $v(t)$ and $w(t)$, and a non-slow variable $x(t)$ with initial conditions $x_0 = \frac{\sqrt{3/5}}{2}, y_0=0, \dot{x}_0 = \frac{1}{5}, \dot{y}_0 = \frac{1}{10}$ and $\epsilon = 0.1$.}
     \label{fig:vwgraphs}
\end{figure}

\subsection{Main results}

The main theoretical result is Theorem\,\ref{mainth}. It demonstrates that the iterated Poincaré map of the slow variables \(u, w\) with respect to the manifold \(y=0\) satisfies the recursive formula \eqref{unwn00}, \eqref{formphin} (in real variables \eqref{sincos00}), which is valid for \(n\) such that \(n\epsilon^3 \ll 1\). For \(n\) slightly smaller, such that \(n\epsilon^{5/2} \ll 1\), the recursion decouples, and we obtain the simpler formula \eqref{unwnsmaller100}; this \(n\) is beyond the range of validity of the perturbation series.

Using the iterated Poincaré map given by Theorem\,\ref{mainth}, the time evolution at any time \(t\) on the \(n+1\)-th fast cycle (where \(n\) is the integer part of \(t/T\), with \(T\) given by \eqref{time one loop}) can be obtained by straightforward integration or by using the perturbation series, with initial conditions calculated from \(u_n, w_n, h\).

The proof of Theorem\,\ref{mainth} can be found in section \S\ref{pfth}.

Section \S\ref{numerical} compares the theoretical results of the theorem with actual numerical results, showing excellent agreement, as seen in Figure\,\ref{finalcomp}. 

\begin{figure}[hbt!]
     \centering
     \begin{subfigure}[b]{0.47\textwidth}
         \centering
         \includegraphics[width=\textwidth]{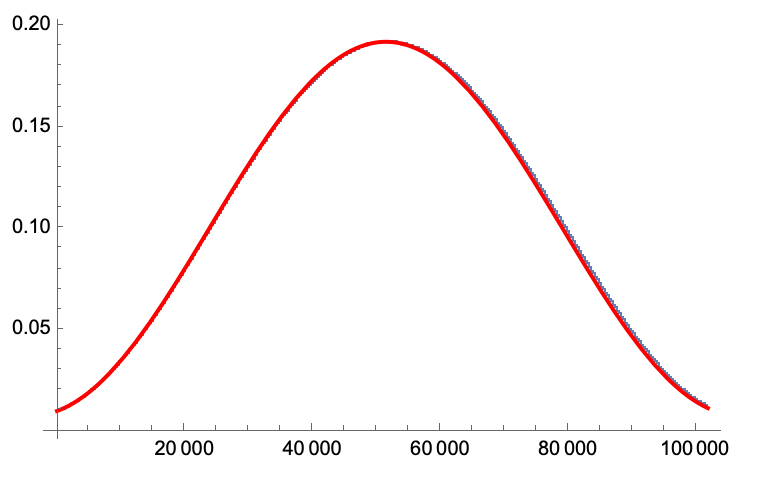}
         \caption{The quantity $v_n$ obtained numerically (blue) and calculated from \eqref{sincos00} (red).}
     \end{subfigure}
     \hfill
     \begin{subfigure}[b]{0.47\textwidth}
         \centering
         \includegraphics[width=\textwidth]{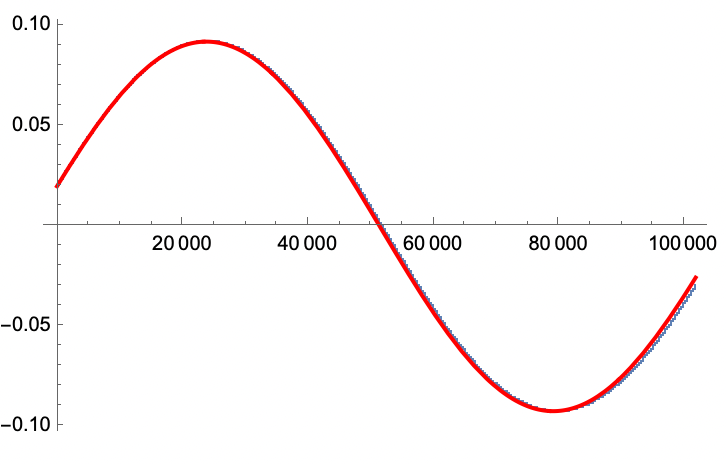}
         \caption{The quantity $w_n$ obtained numerically (blue) and calculated from \eqref{sincos00} (red).}
     \end{subfigure}
     \captionsetup{width = \linewidth}
     \caption{Comparison between numerical calculations and \eqref{sincos00} for $n = 1024000$, $h = 0.1$, $\epsilon = 0.01$,  $x_0 = \frac{\sqrt{3/5}}{2}$, $y_0=0$, $\dot{x}_0= 0.2$, and $\dot{y}_0=0.1$.}
     \label{finalcomp}
\end{figure}

We note that Theorem\,\ref{mainth} assumes the initial condition \(y(0) = 0\), which can always be arranged by allowing the system to evolve until \(y(t) = 0\).

The assumption \eqref{condu0w0} in the Theorem \ref{mainth} simply states that \(x_0 \neq 0\) (since, with \(y_0 = 0\), we have \(h^2 - u_0^2 - w_0^2 = \dot{y}_0^2 \left(2h - \dot{y}_0^2 - \dot{x}_0^2\right) = \dot{y}_0^2 x_0^2 \geq 0\)). In the case \(x_0 = 0\), higher-order expansions are required, which we do not pursue here.

\begin{Theorem}\label{mainth}
Let \(x(t), y(t)\) be solutions of \eqref{perturbedsystem} with initial conditions 
\begin{equation}
\label{indcond}
x(0) = x_0,\ y(0) = 0,\ \dot{x}(0) = \dot{x}_0,\ \dot{y}(0) = \dot{y}_0
\end{equation}
where \(x_0, \dot{x}_0, \dot{y}_0 \in \RR\).

Define \(u = h - \left(y^2 + \dot{y}^2\right)\) and \(w = \dot{x}\dot{y} + xy\), and let \(u_0 := u(0) = h - \dot{y}_0^2\) and \(w_0 := w(0) = \dot{x}_0\dot{y}_0\). Denote by \(u_n, w_n\) the iterated Poincaré map with respect to the manifold \(\{(x, y, \dot{x}, \dot{y}) | y = 0\}\)  (see, e.g., \cite{Coddington}).

(i) If \(\dot{y}_0 = 0\), then \(w_n = O(n\epsilon^3)\) and \(u_n = O(n\epsilon^4)\).

(ii) If \(\dot{y}_0 \neq 0\) and \(u_0 = w_0 = 0\), we have \(u_n = O(n\epsilon^3)\) and \(w_n = O(n\epsilon^3)\).

(iii) If \(\dot{y}_0 \neq 0\), and \(u_0 \neq 0\) or \(w_0 \neq 0\), satisfying 
\begin{equation}
\label{condu0w0}
u_0^2 + w_0^2 < h^2,
\end{equation}
then there exist positive constants \(\epsilon_0\), \(K_0\), and \(M\), which depend only on \(u_0^2 + w_0^2\), such that for any \(N\) satisfying \(N \epsilon_0^3 \leq K_0\), the following holds: for all \(n = 0, 1, \dots, N\) and all \(\epsilon \in [0, \epsilon_0]\), we have
\begin{equation}
\label{solrecTp}
u_n^2 + w_n^2 = u_0^2 + w_0^2 + n\epsilon^3\delta_n
\end{equation}
where \(|\delta_n| \leq M\).

Furthermore, with \(\phi_0\) given by 
\begin{equation}
\label{defphi0}
e^{i\phi_0} = \frac{u_0 + iw_0}{\sqrt{u_0^2 + w_0^2}},
\end{equation}
we have
\begin{equation}
\label{unwn00}
u_n + iw_n = \sqrt{u_0^2 + w_0^2}\, e^{i\phi_n} \left(1 + n\epsilon^3\eta_n\right)
\end{equation}
with
\begin{equation}
\label{formphin}
\phi_n = \phi_0 +  \dfrac{14 \pi}{3}\epsilon^2\sum_{k=0}^{n-1}\sqrt{h^2 - (u_k^2 + w_k^2)}
\end{equation}
and
\begin{equation}
\label{estimetan}
|\eta_n| \leq M.
\end{equation}

For \(n\) slightly smaller, such that \(n\epsilon^{5/2} \ll 1\), and for sufficiently small \(\epsilon\), Formula \eqref{formphin} simplifies to 
\begin{equation}
\label{unwnsmaller100}
\phi_n = \phi_0 +  \dfrac{14 \pi}{3} n\epsilon^2\sqrt{h^2 - (u_0^2 + w_0^2)} + n^2\epsilon^5\eta_n',
\end{equation}
where \(\eta_n'\) is bounded by constants depending only on \(u_0^2 + w_0^2\).

\end{Theorem}

\begin{Remark}
Since \(\phi_n\) is real, separating the real and imaginary parts in \eqref{unwn00} gives
\begin{equation}
\label{sincos00}
\begin{array}{l}
u_n = \sqrt{u_0^2 + w_0^2}\, \cos \phi_n + n\epsilon^3\delta'_{1,n},\\
\\
w_n = \sqrt{u_0^2 + w_0^2}\, \sin \phi_n + n\epsilon^3\delta'_{2,n}, \\
\end{array}
\end{equation}
where \(\delta_{j,n}, \delta'_{j,n}\) are bounded by constants depending only on \(u_0^2 + w_0^2\).

The exponential form in \eqref{unwn00} shows the errors written in multiplicative form.
\end{Remark}

\begin{Remark}
 
  Note that the value of \(n\) for which  \(n\epsilon^{5/2} \ll 1\) can be significantly larger than the value at which secular terms become significant, as the latter corresponds to \(n\epsilon^2 = O(1)\).

\end{Remark}

 
\section{Proof of Theorem\,\ref{mainth}}\label{pfth}

The proof of (i) is found in \S\ref{Pfofi}. The rest of this section is dedicated to the proof of (ii) and (iii) and therefore we assume that $\dot{y}_0\ne 0$ (hence $v_0>0$).
 We need the following preparations and several lemmas, based on which the proof is concluded in \S\ref{Pfofii}, \S\ref{Pfofiii} and \S\ref{Pfofiii2}.

\subsection{Preparation}\label{preparation}

We now revert to the variable \(v\), as defined in \eqref{vw} at the beginning of this section. (Recall that \(u = h - v\).)

 The system \eqref{perturbedsystem} naturally extends to the complex domain, and solutions are analytic {in $t,\epsilon$ and initial conditions}.

We change variables (in the complex domain): pass from the variables $x,\dot{x},y,\dot{y}$ (dependent) and $t$ (independent) to the variables
$h, v, w, t $ (dependent) and $y$ (independent).

In these new variables  \eqref{perturbedsystem} becomes 
\begin{equation}
\label{system hvwt}
\dfrac{dt}{dy} = \dfrac{1}{\sqrt{v-y^2}}, \ \ \dfrac{dh}{dy} = 0,\  \ \dfrac{dv}{dy} = -2\epsilon(x^2-y^2), \ \  \dfrac{dw}{dy} = -\epsilon \left(\dfrac{x^2-y^2}{\sqrt{v-y^2}} \,\Dot{x}+ 2xy \right),
\end{equation}
where  $x=x(y,v,w)$ is given by \footnote{Since $h$ is a  constant we omit marking the dependence on it.}
\begin{equation}
\label{formulax}
x=x(y,v,w;\epsilon)=\frac {1}{ v+2\epsilon y\left( v-{y}^{2}\right) }\left( \,wy+
\sqrt {v-{y}^{2}} \sqrt {\mathcal{S}} \right)
\end{equation}
with
\begin{equation}
\label{formS}
\mathcal{S}=
2\,hv-{v}^{2}-{w}^{2}+2\epsilon y\left(2\,hv-{v}^{2}-{w}^{2}\right)+\frac43\,{
\epsilon}^{2}y^4(v-y^2)-4\,\epsilon\, \left( h- \frac{2}{3}\,v \right) {y}^{3}
\end{equation}
and
\begin{equation}
\label{formuladx}
\dot{x}=\dot{x}(y,v,w;\epsilon)=\frac{w-x(y,v,w;\epsilon)y}{\sqrt{v-y^2}}.
\end{equation}
The initial conditions \eqref{indcond} become
\begin{equation}\label{incondvw}
v_0:=v|_{y=0}=\dot{y}_0^2,\ \ w_0:=w|_{y=0}=\dot{x}_0\dot{y}_0,\  \ t|_{y=0}=0,\ \ h|_{y=0}=h
\end{equation}
We see in \eqref{system hvwt} that the derivatives of $v$ and $w$ are of order $\epsilon$: they are slow variables.

Solutions of \eqref{system hvwt} with initial conditions \eqref{incondvw} satisfy the system of integral equations:
\begin{align}
\label{integralsystem}
\begin{split}
v(y) &=   v_0 + \epsilon \int_{0}^{y} F(s, v(s),w(s),\epsilon)ds, \ \ \ \ \ \ \text{where }\ \ F(s, v(s),w(s),\epsilon) = 2(-x(s)^2+s^2)\\\
w(y) &=    w_0 + \epsilon \int_{0}^{y} G(s,v(s),w(s),\epsilon) ds, \ \ \ \text{where }\ \  G(s,v(s),w(s),\epsilon) = -\left[ \dfrac{(x(s)^2 - s^2)\dot{x}(s)}{\sqrt{v(s)-s^2}} + 2x(s)s \right]\\
\end{split}
\end{align}
and
\begin{align}
\label{integralsystemth}
\begin{split}
t(y) &= \int_{0}^{y} \dfrac{ds}{\sqrt{v(s)-s^2}}, \\
h(y) &= h,
\end{split}
\end{align}
Here, \(x(s) := x(s, v(s), w(s); \epsilon)\) is given by \eqref{formulax}, and \(\dot{x}(s) := \dot{x}(s, v(s), w(s); \epsilon)\) is given by \eqref{formuladx}. The path of integration in the complex \(y\)-plane is chosen to ensure the observed \(2\pi\)-periodicity up to \(O(\epsilon)\) and maintain analyticity.
As a motivation of our choice, note that expanding in series in $\epsilon$ the system \eqref{integralsystem}, we see that   $v(y)=v_0+O(\epsilon)$ therefore $t(y)=\displaystyle \int_0^y\frac1{\sqrt{v_0-s^2}}ds+O(\epsilon)$ and the path of integration should go around both singularities $s=\pm\sqrt{v_0}$ of integrand. \newline 
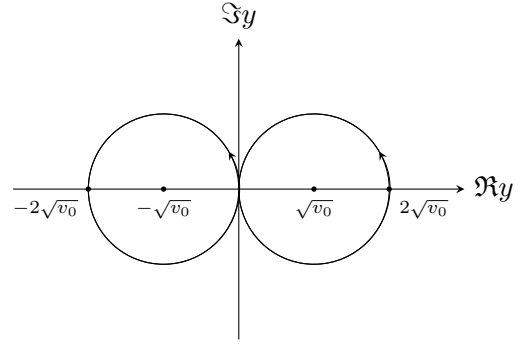
\begin{wrapfigure}[12]{r}{0.4\textwidth}
\vspace{-50pt}
\begin{tikzpicture}[>=stealth]
    
    \draw[->] (-3,0) -- (3,0) node[right] {$\Re y$};
   
    \draw[->] (0,-2) -- (0,2) node[above] {$\Im y$};
    
    \draw (1,0) circle ({sqrt(1)});
    \draw (-1,0) circle ({sqrt(1)});
    
    \fill (1,0) circle (1pt) node[below, font=\tiny] {$\sqrt{v_0}$};
    \fill (-1,0) circle (1pt) node[below, font=\tiny] {$-\sqrt{v_0}$};

    \fill (2,0) circle (1pt) node[below right, font=\tiny] {$2\sqrt{v_0}$};
    \fill (-2,0) circle (1pt) node[below left, font=\tiny] {$-2\sqrt{v_0}$};

    \draw[->] ({1 + cos(30)},{sin(30)}) arc (30:360+30:{1});
    \draw[->] ({-1 + cos(30)},{sin(30)}) arc (30:360+30:{sqrt(1)});
\end{tikzpicture}
\caption{The curve of integration $\mathcal{C}_{v_0}$}
\label{curve}
\end{wrapfigure}
To be precise, we consider the path of integration to be the curve $\mathcal{C}_{v_0}$ defined as follows: starting at $y=0$ the path goes counterclockwise along a circle centered at $\sqrt{v_0}$ and of radius $\sqrt{v_0}$ followed, counterclockwise, by the circle centered at $y = -\sqrt{v_0}$ and of radius $\sqrt{v_0}$, see figure\,\ref{curve}. 

Integrating once around $\mathcal{C}_{v_0}$ we obtain the Poincar\'e map. Denote by $v_1,w_1$ the values of $v(y),w(y)$ after integrating once along $\mathcal{C}_{v_0}$:\vskip 0.5cm
\begin{equation}\label{v1w1}
    v_1=v_0 + \epsilon \oint_{\mathcal{C}_{v_0}} F(s, v(s),w(s),\epsilon)ds, 
w_1=  w_0 + \epsilon \oint_{\mathcal{C}_{v_0}} G(s,v(s),w(s),\epsilon) ds
\end{equation}
and, the time to complete a loop is
$$ t_1 = \oint_{\mathcal{C}_{v_0}} \dfrac{ds}{\sqrt{v(s)-s^2}}
$$

Continuing integration along subsequent loops $\mathcal{C}_{v_0}$  we obtain an iterated  Poincar\'e map.  Along  the first loop $v(y)$ and $w(y)$ are determined as solutions of the integral system $v(y)=v_0 + \epsilon \displaystyle\int_0^y F(s, v(s),w(s),\epsilon)ds$ and $w(y)=  w_0 + \displaystyle\epsilon \int_0^y G(s,v(s),w(s),\epsilon) ds$ which is proven to have a unique solution in Lemma\,\ref{existence}.
With these values for $v(y),w(y)$, at the end of the first loop we have \eqref{v1w1}. Continuing the integration and applying Lemma\,\ref{existence} on subsequent loops, we find that on the \((n+1)\)-th loop, \(v(y)\) and \(w(y)\) are the unique solutions of the integral system:
\begin{equation}\label{sysvwny}
v(y)=v_n + \epsilon \int_0^y F(s, v(s),w(s),\epsilon)ds, \ \ \ \
w(y)=  w_n + \epsilon \int_0^y G(s,v(s),w(s),\epsilon) ds
\end{equation}
Here, we denote by \(v_{n}\) and \(w_{n}\) the values of \(v\) and \(w\) at the end of the \(n\)-th loop. Consequently, at the end of the \((n+1)\)-th loop, we have:
\begin{equation}\label{vnloop}
v_{n+1}=v_n + \epsilon \oint_{\mathcal{C}_{v_0}} F(s, v(s),w(s),\epsilon)ds, \ \ \ \
w_{n+1}=  w_n +\epsilon \oint_{\mathcal{C}_{v_0}} G(s,v(s),w(s),\epsilon) ds
\end{equation}

\subsection{Epsilon-expansion with remainder}
We expand $v(y)$ and $w(y)$ in $\epsilon$ up to $O(\epsilon^3)$ and keep track of remainders. Substituting
\begin{equation}
    \label{vwhigherExpansion}
    \begin{split}
        v(y) &= v^{[0]}(y)+\epsilon v^{[1]}(y) + \epsilon^2 v^{[2]}(y) + \epsilon^3 R(y), \\
    w(y) &= w^{[0]}(y)+\epsilon w^{[1]}(y) + \epsilon^2 w^{[2]}(y) + \epsilon^3 S(y).
    \end{split}
\end{equation}
in \eqref{integralsystem} (where the remainders $R,S$ also depend on $\epsilon$), expanding and identifying the powers of $\epsilon$ we obtain
$$v^{[0]}(y) = v_0,\ \ \ w^{[0]}(y) = w_0$$
and
\begin{align*}
    v^{[1]}(y)&=\int_0^yF(s,v_0,w_0,0)\, ds, \ \ \ \ w^{[1]}(y)=\int_0^yG(s,v_0,w_0,0)\, ds, \\
    v^{[2]}(y)&= \int_0^y\left[F_v(s,v_0,w_0,0)v^{[1]}(s)+F_w(s,v_0,w_0,0)w^{[1]}(s)+F_\epsilon(s,v_0,w_0,0)\right]\, ds, \\
    w^{[2]}(y)&= \int_0^y\left[G_v(s,v_0,w_0,0)v^{[1]}(s)+G_w(s,v_0,w_0,0)w^{[1]}(s)+G_\epsilon(s,v_0,w_0,0)\right]\, ds.
\end{align*}
from which the terms $v^{[1]}(y)$, $w^{[1]}(y)$, $v^{[2]}(y)$, and $w^{[2]}(y)$ are calculated explicitly; they are given in the appendix \ref{vwExpansions}.

 \subsubsection{Remainders} To derive the equations satisfied by the remainders \(R\) and \(S\), we rewrite \eqref{vwhigherExpansion} in the form:
\begin{equation}\label{defxi}
    v(y) = v_0 + \epsilon \xi_v(y), \ \ \ \ w(y) = w_0 + \epsilon \xi_w(y),
\end{equation}
where $\xi_v(y) = v^{[1]}(y) + \epsilon v^{[2]}(y) + \epsilon^2 R(y)$ and $\xi_w(y) = w^{[1]}(y) + \epsilon w^{[2]}(y) + \epsilon^2 S(y)$. Using \eqref{defxi}, the functions \(F\) and \(G\) (with explicit formulas provided in \eqref{integralsystem}) admit \(\epsilon\)-expansions with remainders:
\begin{equation} \label{ExpansionFG}
    \begin{split}
        F(y,v_0 + \epsilon \xi_v(y),w_0 + \epsilon \xi_w(y),\epsilon) &= \overset{\circ}{F}+\epsilon\xi_v \overset{\circ}{F_v}+\epsilon\xi_w\overset{\circ}{F_w}+\epsilon\overset{\circ}{F_\epsilon}
    +\epsilon^2 \mathcal{Q}_F, \\
    G(y,v_0 + \epsilon \xi_v(y),w_0 + \epsilon \xi_w(y),\epsilon) &= \overset{\circ}{G}+\epsilon\xi_v \overset{\circ}{G_v}+\epsilon\xi_w\overset{\circ}{G_w}+\epsilon\overset{\circ}{G_\epsilon}
    +\epsilon^2 \mathcal{Q}_G,
    \end{split}
\end{equation}
where $\overset{\circ}{f}$ denotes the function $f$ evaluated at $(y,v_0,w_0,0)$ and $\mathcal{Q}_F$, $\mathcal{Q}_G$ are quadratic forms in $\langle \xi_v,\xi_w,1\rangle$ whose coefficients are integrals, in $\lambda$ on $[0,1]$, of $1-\lambda$ multiplying second derivatives of $F$ evaluated at $(y, v_0+\lambda \epsilon\xi_v,w_0+\lambda \epsilon\xi_w, \lambda\epsilon)$; hence $ \mathcal{Q}_F= \mathcal{Q}_F(y,\epsilon,\epsilon^2 R,\epsilon^2 S)$,  $ \mathcal{Q}_G= \mathcal{Q}_G(y,\epsilon,\epsilon^2 R,\epsilon^2 S)$ and they are analytic in $\epsilon$.

Substituting \eqref{vwhigherExpansion}, \eqref{ExpansionFG} in \eqref{integralsystem} we obtain integral equations for $R,S$:
\begin{equation}\label{eqRS}
    \begin{split}
        R(y) &= \int_0^y [f_1(s)+ \epsilon f_2(s,\epsilon R(s), \epsilon S(s)) ] ds:=\mathcal{I}_1(R,S)(y), \\
        S(y) &= \int_0^y [g_1(s)+ \epsilon g_2(s,\epsilon R(s), \epsilon S(s)) ] ds:= \mathcal{I}_2(R,S)(y).
    \end{split}
\end{equation}
where 
\begin{equation}
\label{formf12}
\begin{array}{l}
f_1=\frac12 \overset{\circ}{F_{vv}}v^{[1]}\,^2+\frac12 \overset{\circ}{F_{ww}}w^{[1]}\,^2+ \overset{\circ}{F_{vw}}v^{[1]}w^{[1]}+\overset{\circ}{F_{v\epsilon}}v^{[1]}+\overset{\circ}{F_{w\epsilon}}w^{[1]}+\frac12 \overset{\circ}{F_{\epsilon\epsilon}},   \\ \\
f_2= \overset{\circ}{F_v}\left( v^{[2]}+\epsilon R\right)+\overset{\circ}{F_w}\left( w^{[2]}+\epsilon S\right)+
 \mathcal{Q}_F(y,\epsilon,\epsilon^2 R,\epsilon^2 S).
\end{array}
 \end{equation}
The functions $g_1,g_2$ are similar.

Existence and uniqueness of solutions $R(y),S(y)$ of the integral system \eqref{eqRS} where the integral goes along $\mathcal{C}_{v_0}$ is proved in Lemma\,\ref{L2}.

\subsubsection{Values after one loop}
Let us denote the values of $v^{[1]}(y)$, $w^{[1]}(y)$, $v^{[2]}(y)$, and $w^{[2]}(y)$ at $y = 0$ after a full loop along $\mathcal{C}_{v_0}$ by $AC v^{[1]}(0)$, $AC w^{[1]}(0)$, $AC v^{[2]}(0)$, and $AC w^{[2]}(0)$, respectively. A direct inspection of the expressions in appendix \ref{vwExpansions} shows that
\begin{align*}
    AC v^{[1]}(0) &= v^{[1]}(0)=0, \  \ \ 
    AC w^{[1]}(0) = w^{[1]}(0)=0, \\
    AC v^{[2]}(0) &= \dfrac{14 \pi}{3} w_0 \sqrt{2hv_0 - v_0^2 - w_0^2}, \\
    AC w^{[2]}(0) &= \dfrac{14 \pi}{3} (h-v_0) \sqrt{2hv_0 - v_0^2 - w_0^2}.
\end{align*}
Therefore, after one loop the values at $y=0$ of $v$ and $w$ in \eqref{vwhigherExpansion} become
\begin{equation}\label{loopone}
\begin{split}
v_1:=&v_0+\epsilon^2\, \dfrac{14 \pi}{3} w_0 \sqrt{2hv_0 - v_0^2 - w_0^2}+\epsilon^3 R_0,\\
w_1:=&w_0+\epsilon^2\,\dfrac{14 \pi}{3} (h-v_0) \sqrt{2hv_0 - v_0^2 - w_0^2}+\epsilon^3S_0
\end{split}
\end{equation}
where $R_0,S_0$ denote the values of $R,S$ after one loop.

\begin{Remark}\label{remarkreal}
    We use real initial conditions. After integration in the complex plane, along $\mathcal{C}_{v_0}$, the values $v_1$, $w_1$ are, again, real numbers (as they must coincide with the values of the solution obtained in the real domain). 
\end{Remark}

\subsection{Heuristics}
As long as  we can continue integrating on subsequent loops integration, (i.e., as long as Lemmas\,\ref{existence} and\,\ref{L2} remain applicable), after the \((n+1)\)-th loop, we have \eqref{vnloop}. There, using \eqref{loopone} for the \((n+1)\)-th loop, we obtain:
\begin{equation}
\label{recurrencevwR}
\begin{split}
v_{n+1} - v_n &= \dfrac{14 \pi}{3}  \epsilon^2\,w_n \sqrt{2hv_n -v_n^2 - w_n^2}\ + \epsilon^3R_n \\
w_{n+1} - w_n &= \dfrac{14 \pi}{3} \epsilon^2\, (h-v_n)\sqrt{2hv_n - v_n^2 - w_n^2}\  + \epsilon^3 S_n
\end{split}
\end{equation}
To solve \eqref{recurrencevwR}, approximately, to order $O(\epsilon^3)$, it is natural to look at the differential system  
\begin{equation}
\begin{split}
\dfrac{dV}{dn} &=  \dfrac{14 \pi}{3}\, \epsilon^2\,W \sqrt{2hV -V^2 - W ^2} , \\
\dfrac{dW}{dn} &= \dfrac{14 \pi}{3} \, \epsilon^2\, (h-V) \sqrt{2hV - V^2 - W^2}  
\end{split}
\end{equation}
which has the solutions \begin{equation}
\label{tvntwn}
\begin{split}
V(n) &= h + A \sin(n \beta \epsilon^2 + B), \\
W(n) &= A \cos(n \beta \epsilon^2 + B), \ \ \ \text{where } \beta=\frac{14 \pi \sqrt{k_0}}{3}
\end{split}
\end{equation}
where the constant quantity $k_0 = 2hV-V^2-W^2$ and the constants of integration $A,\, B$ are determined from the initial conditions: $V(0)=v_0,\ W(0)=w_0$.

\
In the following, we prove that this  picture is correct: the iteration can be continued for \(n\) as stated in Theorem\,\ref{mainth}, that \(V(n)\) and \(W(n)\) indeed approximate \(v_n\) and \(w_n\), respectively, and we provide error estimates.
\subsection{Solutions along one loop}
Lemma\, \ref{existence} proves the existence, form and bounds of solutions of \eqref{integralsystem} along one loop. Lemma\,\ref{L2} gives estimates for the $O(\epsilon^3)$ remainders. Since $v(y)$ and $w(y)$ will have different values at the end of the loop $\mathcal{C}_{v_0}$, let us denote by $\overline{\mathcal{C}}_{v_0}$ the open path obtained from $\mathcal{C}_{v_0}$ by making the starting point $y=0$ and the endpoint (also $y=0$) distinct.

\subsection{Solutions Along One Loop}

Lemma\,\ref{existence} establishes the existence, form, and bounds of solutions to \eqref{integralsystem} along a single loop. Lemma\,\ref{L2} provides estimates for the \(O(\epsilon^3)\) remainders. Since \(v(y)\) and \(w(y)\) take on different values at the end of the loop \(\mathcal{C}_{v_0}\), we denote by \(\overline{\mathcal{C}}_{v_0}\) the open path derived from \(\mathcal{C}_{v_0}\)  by making the starting point $y=0$ and the endpoint (which is also $y=0$) distinct.

\begin{lemma} \label{existence} 
Fix \(h > 0\). Assume \(v_0 > 0\) and \(w_0 \in \mathbb{R}\) satisfy\footnote{These are the assumptions of Theorem\,\ref{mainth} (iii), since \(u = h - v\).}
\begin{equation}
    \label{initialconditionsassumptions0}
    (h - v_0)^2 + w_0^2 < h^2,
\end{equation}
and let \(c_0\) be a positive constant such that
\begin{equation}
    \label{initialconditionsassumptions}
    (h - v_0)^2 + w_0^2 < c_0^2 < h^2.
\end{equation}

Then the solution of the system \eqref{integralsystem} with initial conditions \((v_0, w_0)\) satisfies the following:

\begin{enumerate}[label=(\roman*)]
    \item The system of integral equations \eqref{integralsystem} has a unique solution for \(y\) surrounding \(\overline{\mathcal{C}}_{v_0}\) once, provided \(\epsilon\) is sufficiently small: \(0 < \epsilon \leq \epsilon_0\), where \(\epsilon_0\) depends only on \(c_0\).
    
    More precisely, the solution takes the form \(v(y) = v_0 + \epsilon \tilde{v}(y)\), \(w(y) = w_0 + \epsilon \tilde{w}(y)\), with \(|\tilde{v}(y)| \leq \tilde{M}\) and \(|\tilde{w}(y)| \leq \tilde{M}\), where \(\tilde{M}\) is a constant depending only on \(c_0\). 
    
    \item The solution is continuous with respect to the parameters \(v_0\), \(w_0\), and \(\epsilon\) for \(v_0, w_0\) in the domain defined by \eqref{initialconditionsassumptions} and \(\epsilon\) satisfying \(0 \leq \epsilon \leq \epsilon_0\). Furthermore, the solution is analytic in \(\epsilon\) at \(\epsilon = 0\).
\end{enumerate}
\end{lemma}
\begin{proof} \label{PfL1}

{
We focus on the integral equations for $v$ and $w$ in \eqref{integralsystem}, as the behavior of $t$ is an immediate consequence (see Corollary \ref{timeInt}). These equations are
    \begin{equation}
        \label{eqvw}
        v(y) = v_0 + \epsilon \int_0^y F(s,v(s),w(s),\epsilon)ds, \ \ w(y) = w_0 + \epsilon \int_0^y G(s,v(s),w(s),\epsilon) ds
    \end{equation}
    where the integrals are calculated along $\mathcal{C}_{v_0}$.} With the substitution $v(y)=v_0+\epsilon\tilde{v}(y), w(y)=w_0+\epsilon\tilde{w}(y)$,
equations  \eqref{integralsystem} become
\begin{equation}
\label{tildeeq}
    \begin{split}
        \tilde{v}(y)&= \displaystyle \int_{0}^y F(s,v_0+\epsilon\tilde{v}(s),w_0+\epsilon\tilde{w}(s),\epsilon)\, ds, \\
        \tilde{w}(y)&= \displaystyle \int_{0}^y G(s,v_0+\epsilon\tilde{v}(s),w_0+\epsilon\tilde{w}(s),\epsilon)\, ds.
    \end{split}
\end{equation}
 Consider the Banach space $\mathcal{B}$ of pairs $\tilde{\mathbf{v}}(y)=\langle\tilde{v}(y),\tilde{w}(y)\rangle$ of continuous functions on $\overline{\mathcal{C}}_{v_0}$, with the norm $\|\tilde{\mathbf{v}}\|=\max\{\sup_{\overline{\mathcal{C}}_{v_0}} |\tilde{v}(y)|, \sup_{\overline{\mathcal{C}}_{v_0}} |\tilde{w}(y)|\}$.
\newline
Define $J:\mathcal{B} \rightarrow \mathcal{B}$ by
\begin{equation*}
    J(\tilde{\mathbf{v}}) = J(\langle\tilde{v},\tilde{w}\rangle) := \langle J_1(\tilde{v},\tilde{w}), J_2(\tilde{v},\tilde{w}) \rangle
\end{equation*}
where 
\begin{equation}\label{j1j2}
\begin{split}
    J_1(\tilde{v},\tilde{w}) &:= \int_{0}^y F(s,v_0+\epsilon\tilde{v}(s),w_0+\epsilon\tilde{w}(s),\epsilon)\, ds, \\
    J_2(\tilde{v},\tilde{w}) &:= \int_{0}^y G(s,v_0+\epsilon\tilde{v}(s),w_0+\epsilon\tilde{w}(s),\epsilon)\, ds.
\end{split}
\end{equation}

We show that $J$ is a contraction in a suitable ball $\|\tilde{\mathbf{v}}\|\le \tilde{M}$ in $\mathcal{B}$, implying  that $\tilde{\mathbf{v}}=J\tilde{\mathbf{v}}$ has a unique solution, by contraction mapping theorem.

Note that by \eqref{integralsystem} we have
\begin{equation}
\label{F120}
F(y,v_0,w_0,0)= {2} y^2-{2} (x^{[0]})^2:=F_1(y),\ \ \ \  G(y,v_0,w_0,0)=\frac{y^2-(x^{[0]})^2}{\sqrt{v_0-y^2}}\, \dot{x}^{[0]}\,-2x^{[0]}\,y:=F_2(y)
\end{equation}
where, from \eqref{formulax}, \eqref{formS}, \eqref{formuladx}
\begin{equation}
\label{S0}
x^{[0]}\,=\frac{w_0y+\sqrt{v_0-y^2}\sqrt{\mathcal{S}_0}}{v_0}\ \text{with\ }\ \mathcal{S}_0=2\,hv_0-{v}_0^{2}-{w}_0^{2},\ \text{ and }\ \ \ \dot{x}^{[0]}\,=\frac{w_0-x^{[0]}\,y}{\sqrt{v_0-y^2}}.
\end{equation}

{\em{I. The constant $\tilde{M}$.}} 

We first show that the functions \eqref{F120} are bounded for $y\in{\mathcal{C}_{v_0}}$ with a bound depending only on $c_0$ (and not on each $v_0,w_0$ separately). To see this, note that 
 
(a) If $(v_0,w_0)$ satisfy \eqref{initialconditionsassumptions} then necessarily 
\begin{equation}
\label{eqA}  
     0<h-c_0 \le v_0\le h+c_0, \text{ and }|w_0| \le h.
\end{equation}

(b) For $y\in{\mathcal{C}_{v_0}}$, we have
\begin{equation}
\label{eqB}
v_0\le |v_0-y^2|\le 3v_0
\end{equation}
and
\begin{equation}
\label{maxy}
 |y|  \le 2\sqrt{v_0} \le 2 \sqrt{h+c_0}
 \end{equation}
(these follow by noting that such a $y$ has the form $y=\sqrt{v_0}(\pm 1+e^{i\theta})$ for some $\theta$, and thus $ \left|v_0-y^2\right| = {v_0}\, \left|\pm2+e^{i\theta}\right|$).

 Thus, from \eqref{eqA}, \eqref{eqB} it follows that $0<h-c_0\le  v_0\le |v_0-y^2|\le 3(h+c_0)$ and $ \mathcal{S}_0  = h^2 - (h-v_0)^2 - w_0^2$. Thus
 \begin{equation}
 \label{boundsS0}
  0<h^2-c_0^2\le\mathcal{S}_0\le h^2.  
   \end{equation}
 
 So, 
 \begin{align*}
 &|x^{[0]}| \le \dfrac{h(2\sqrt{h+c_0}) + \sqrt{3}\sqrt{h+c_0}\  h}{h-c_0} :=c_2 \\
& |\dot{x}^{[0]}| \le \dfrac{h + 2 \sqrt{h+c_0}\ c_2  }{\sqrt{h-c_0}}:=c_3
 \end{align*}
 Thus, the bounds on $F_1$ and $F_2$ depend only on $c_0$:
 \begin{align*}
     &|F_1(y)| = |2y^2-2x^{[0]}\,^2| \le 2|y^2| + 2|x^{[0]}\,|^2 \le 8(h+c_0) + 2 c_2^2 \\
     &|F_2(y)| \le \dfrac{|y^2| + |x^{[0]}\,|^2}{|\sqrt{v_0-y^2}|} |\dot{x}^{[0]}|\, + 2|x^{[0]}|\ |y| \le \left(\dfrac{4(h+c_0) + c_2^2}{\sqrt{h-c_0}}\right) c_3 + 4c_2\sqrt{h+c_0} 
 \end{align*}

Therefore, for $y\in \mathcal{C}_{v_0}$ we have $\left|\displaystyle \int_{0}^y F_{1,2}(s)\, ds \right|\le 4\pi\sqrt{v_0}\max |F_{1,2}|\le 4\pi \sqrt{h+c_0} \max|F_{1,2}| =:K_0$. Clearly, $K_0$ is a constant that depends only on $c_0$. Let $\tilde{M}=2K_0$.

{\em{II:}}  We now show that the function $J$ leaves the ball $\|\tilde{\mathbf{v}}\| \le \tilde{M}$ invariant. Let $\tilde{\mathbf{v}}$ be such that $\|\tilde{\mathbf{v}}\|  =\| \langle \tilde{v}(y), \tilde{w}(y)  \rangle \|  \le \tilde{M}$. 

By substituting $v(y) = v_0 + \epsilon \Tilde{v}$, $w(y) = w_0 + \epsilon \Tilde{w}$, write the quantity \eqref{formS} as $\mathcal{S}=\mathcal{S}_0+\epsilon \mathcal{S}_1$ where $\mathcal{S}_0 = 2\,hv_0-{v}_0^{2}-{w}_0^{2}$ and $\mathcal{S}_1$ is a polynomial in $\epsilon,v_0, \tilde{v}, w_0, \tilde{w},y$. We know that $\mathcal{S}_0$ satisfies \eqref{boundsS0}.

Let $0 < \epsilon \le \epsilon_0 $ be small enough (with $\epsilon_0\le 1$)\footnote{The maximal value $1$ was chosen conventionally, for the estimate \eqref{epsilonthree} below.} so that it satisfies the following: 
\begin{enumerate}
    \item $\epsilon \tilde{M}\le \alpha(h-c_0)\ \ \text{\ for some }\alpha\in(0,1)$, \label{epsilonone}
    \item $(h-c_0)-\epsilon \left[\tilde{M} + 4 \sqrt{h+c_0}(3(h+c_0)+\epsilon \tilde{M})  \right]  > \frac12(h-c_0)>0$, and \label{epsilontwo}
    \item $( h^2-c_0^2)-\epsilon \max_{0\le\epsilon\le1}\max_{\{v_0,w_0\,|\,(h-v_0)^2+w_0^2\le c_0^2\} }\max_{ \|\tilde{\mathbf{v}}\|\le \tilde{M},y\in\mathcal{C}_{v_0}}  \,|\mathcal{S}_1|\ >\frac12( h^2-c_0^2)>0$ \label{epsilonthree}
\end{enumerate}

We now establish that the denominators which appear in the integrands in \eqref{tildeeq} do not vanish for $y$ on the path of integration and, moreover, are  uniformly bounded below.

First, we have, using  \eqref{eqB} and \eqref{epsilonone} from above:   $|v-y^2|=|v_0+\epsilon\tilde{v}(y)-y^2|\ge |v_0-y^2|-\epsilon|\tilde{v}(y)|\ge {v_0}-\epsilon \tilde{M}\ge (1-\alpha)(h-c_0)>0$. 

Then, the denominator in \eqref{formulax} does not vanish: using \eqref{eqA}, \eqref{eqB}, \eqref{maxy} and \eqref{epsilontwo} we have $|v+2\epsilon y(v-y^2)|\ge v_0-\epsilon \tilde{M} -4\epsilon \sqrt{v_0} |v-y^2|\ge v_0-\epsilon \tilde{M} -4\epsilon \,  \sqrt{h+c_0}(3v_0+\epsilon \tilde{M})    \ge h-c_0-\epsilon \tilde{M} -4\epsilon \sqrt{h+c_0}(3(h+c_0)+\epsilon \tilde{M})\ge \frac12(h-c_0) >0$.

Finally, the quantity $\mathcal{S}$ in \eqref{formS} does not vanish since $|\mathcal{S}|\ge |\mathcal{S}_0|-\epsilon |\mathcal{S}_1|\ge h^2-c_0^2-\epsilon \max|\mathcal{S}_1|>0$ where for the last inequality we used \eqref{epsilonthree} from above.

We claim that the functions $F, G,\partial_vF, \partial_vG,\partial_wF, \partial_wG,  \partial_\epsilon F, \partial_\epsilon G$ evaluated at $(y,v_0+\epsilon\tilde{v}, w_0+\epsilon\tilde{w},\epsilon)$ are well defined, and continuous if $\|{\tilde{\mathbf{v}}}\|\le \tilde{M}$. 
Indeed, these functions are rational functions of $v,w,y,\epsilon, \sqrt{\mathcal{S}}$, $\sqrt{v-y^2},y$; their denominators are products of $ v+2\epsilon y\left( v-{y}^{2}\right) ,\sqrt{\mathcal{S}}, \sqrt{v-y^2}$ which, by the estimates above,  for all $\epsilon\le \epsilon_0$, are bounded below by a positive constant dependent only on $c_0$. 

Then the absolute values of $\partial_vF, \partial_vG,\partial_wF, \partial_wG,  \partial_\epsilon F, \partial_\epsilon G$ evaluated at points $(y, v_0+\epsilon\tilde{\tilde{v}},w_0+\epsilon\tilde{\tilde{w}},\epsilon)$ with $|\tilde{\tilde{v}}|\le \tilde{M}$ and $|\tilde{\tilde{w}}|\le \tilde{M}$ are bounded by some constant $C$ which depends only on $c_0$.

From \eqref{j1j2} we have, with the bound $C$ above,
$$|J_{1,2}(\tilde{\mathbf{v}})|\le K_0+\epsilon \ell\, C\,(2\tilde{M}+1)\le K_0+K_0$$
where $ \ell $ is the length of the loop $\mathcal{C}_{v_0}$ (so $\ell=4\pi\sqrt{v_0}\le 4\pi\sqrt{h+c_0}:=\ell_0$) and the last inequality holds for $\epsilon$ small enough: $\epsilon\le\epsilon_0':=\min\{\epsilon_0,K_0/[\ell_0 C (4K_0+1)]\}$.
Therefore $|J_{1,2}(\tilde{\mathbf{v}})|\le \tilde{M}$ and $J$ leaves the ball invariant.

{\em III. } $J$ is a contraction. Indeed, with $C$ same as above, we have $|J_{1,2}(\tilde{\mathbf{v}})-J_{1,2}(\tilde{\mathbf{v}}')|\le \epsilon 2\ell C\|\tilde{\mathbf{v}}-\tilde{\mathbf{v}}' \|$. Therefore, $J$ is a contraction if $ \epsilon<1/(2\ell _0C):=\epsilon_0''$.

{\em IV. } In conclusion, by the contraction mapping theorem, the system \eqref{tildeeq} has a unique solution for $\epsilon<\min\{\epsilon_0',\epsilon_0''\}$.

(ii) Continuity in parameters and initial conditions follow from general theorems. 
\end{proof}

\begin{corollary}\label{timeInt}
    After integration along a loop $\mathcal{C}_{v_0}$, the time which is initially $0$ becomes, by Lemma \ref{existence}:
    \begin{equation*}
        t_1 = \oint \dfrac{dy}{\sqrt{v(y)-y^2}} = \oint \dfrac{dy}{\sqrt{v_0 + \epsilon \Tilde{v}(y)-y^2}} = \oint \dfrac{dy}{\sqrt{v_0-y^2}} + O(\epsilon) = 2\pi + O(\epsilon)
    \end{equation*}
    therefore $t_1$ is a ``quasi-period" of the motion. In fact we have $t_1 = 2\pi + O(\epsilon^2)$. 
\end{corollary} 

\

Lemma\,\ref{L2} estimates the remainders $R,S$ in \eqref{vwhigherExpansion}; recall that they satisfy the integral system \eqref{eqRS}.

\begin{lemma}\label{L2}

Under the assumptions of Lemma\,\ref{existence} 
there exist positive constants $\epsilon_0$ and $M$ depending only on $c_0$, such that the system \eqref{eqRS} has a unique solution $R,S$ and this solution satisfies $$|R(y)|\le M \ \text{and}\ |S(y)|\le M\ \text{for all }  y\in \overline{\mathcal{C}}_{v_0},\  \epsilon\le \epsilon_0. $$

Moreover, 
 $R,S$ depend continuously on $v_0,w_0,\epsilon$.
\end{lemma}

\begin{proof}
 The proof is similar to the proof of Lemma\,\ref{existence}. Fix $v_0,w_0$ satisfying \eqref{initialconditionsassumptions} and consider the same Banach space $\mathcal{B}$ as in the proof of Lemma\,\ref{existence}.
We show that the operator $(\mathcal{I}_1,\mathcal{I}_2)$ defined in \eqref{eqRS}
is contractive in $\mathcal{B}$ if $\epsilon$ is small enough.

We see that the structure of the integral operators $\mathcal{I}_{1,2}$  in \eqref{eqRS} is similar to that of the integral operators $J_{1,2}$ in the proof of Lemma\,\ref{existence}. One difference is the appearance of the function $\arcsin \tfrac y{\sqrt{v_0}}$ (through the functions $v^{[2]}, w^{[2]}$) but this is regular on $\overline{\mathcal{C}}_{v_0}$. 

The same arguments as in the proof of Lemma\,\ref{existence} go through straightforwardly.

\end{proof}
\subsection{Solution along $n$ loops: iteration}
Under the assumptions of Lemma\,\ref{existence}, after the first loop we replace the initial conditions   $v_0,w_0$ by $v_1,w_1$. 
We can assume that $v_1,w_1$  satisfy \eqref{initialconditionsassumptions} decreasing $\epsilon_0$ and increasing  $c_0$ if needed; this is possible by Lemma \,\ref{existence}. 
Lemma\,\ref{existence} can be applies again, for initial conditions $v_1,w_1$. This procedure can be continued inductively for a finite number $N$ of steps. The question is how large can $N$ be: we find that the condition is $N\epsilon^3=O(1)$.

Let $v_n,w_n$ be the values of $v_0,w_0$ after $n$ loops, see also \S\ref{preparation}; $v_n,w_n$ are real numbers, see Remark\,\ref{remarkreal}. After $n$ loops they satisfy \eqref{recurrencevwR} (where $R_n,\,S_n$ depend on $v_0,w_0,\epsilon$).
 
{\bf Rescaling.} For simplicity, we rescale the variables as follows. Let \(\beta = \dfrac{14 \pi}{3}\), and define 
\begin{equation} \label{rescaling}
    \sqrt{\beta}\epsilon = \epsilon_1, \quad \beta^{-3/2}R_n = R_{n,1}, \quad \beta^{-3/2}S_n = S_{n,1}.
\end{equation}

For the remainder of this section, we omit the subscript \(1\) in the notation above to avoid overburdening the expressions. Recall that \(u = h - v\), and denote \(u_n = h - v_n\).

After rescaling, the recurrence \eqref{recurrencevwR} becomes
\begin{equation}
\label{recurrenceuw}
\begin{split}
    u_{n+1} &= u_n - \epsilon^2 w_n \sqrt{h^2 - u_n^2 - w_n^2} - \epsilon^3 R_n, \\
    w_{n+1} &= w_n + \epsilon^2 u_n \sqrt{h^2 - u_n^2 - w_n^2} + \epsilon^3 S_n. 
\end{split}
\end{equation}
where $R_n,\,S_n$ depend on $u_0,w_0,\epsilon$.

\begin{lemma}
    \label{Tn}
    Fix \(h > 0\) and assume that for some positive constant \(c_0\), we have
    \begin{equation}
    \label{condu0w0c0}
    u_0^2 + w_0^2 < c_0^2 < h^2.
    \end{equation}
    
    Consider the solution of the recurrence \eqref{recurrenceuw} with initial conditions \(u_0, w_0\). Assume that for some positive constants \(\epsilon_0, M\), and for some natural number \(N\), the remainders in \eqref{recurrenceuw} satisfy
    \begin{equation}
    \label{assumeRnSn}
    |R_n| \leq M, \quad |S_n| \leq M, \quad \text{for all } n = 0, 1, \ldots, N-1, \quad \text{and all } \epsilon \in [0, \epsilon_0].
    \end{equation}

    \begin{enumerate}[label=(\roman*)]
        \item If \(u_0 \neq 0\) or \(w_0 \neq 0\), let \(c_1 > 0\) be such that
        \[
        0 < c_1 < u_0^2 + w_0^2.
        \]
        Then, there exist positive constants \(M_\delta, M_\eta,\) and \(K\), depending only on \(M, c_0\), and \(c_1\), such that if
        \begin{equation}
        \label{assumeN}
        N \epsilon_0^3 \leq K,
        \end{equation} 
        then for all \(n = 0, 1, \ldots, N\) and \(\epsilon \in [0, \epsilon_0]\), we have
        \begin{equation}
        \label{solrecTp}
        T_n := u_n^2 + w_n^2 \quad \text{satisfies} \quad T_n = T_0 + n\epsilon^3\delta_n, \quad \text{with } |\delta_n| \leq M_\delta,
        \end{equation}
        and
        \begin{equation}
        \label{unwn}
        u_n + i w_n = \sqrt{T_0}\, e^{i\phi_0 + i\epsilon^2 \sum_{k=0}^{n-1} \sqrt{h^2 - T_k}}\, \left(1 + n\epsilon^3 \eta_n\right),
        \end{equation}
        where \(\phi_0\) is given by
        \begin{equation}
        \label{defphi0}
        \frac{u_0 + i w_0}{\sqrt{T_0}} = e^{i\phi_0},
        \end{equation}
        and
        \begin{equation}
        \label{estimetan}
        |\eta_n| \leq M_\eta.
        \end{equation}

        \item For \(n\) small enough so that \(n \epsilon^{5/2} \ll 1\), and for \(\epsilon\) small enough, Formula \eqref{unwn} simplifies to
        \begin{equation}
        \label{unwnsmaller}
        u_n + i w_n = \sqrt{T_0}\, \exp\left[i\phi_0 + i n\epsilon^2 \sqrt{h^2 - T_0} + i n^2 \epsilon^5 \eta_n'\right]\, \left(1 + n\epsilon^3 \eta_n\right),
        \end{equation}
        where \(\eta_n', \eta_n\) are bounded by constants depending only on \(M, c_0, c_1\). Separating the real and imaginary parts in \eqref{unwnsmaller}, we obtain
        \begin{equation}
        \label{sincos}
        \begin{aligned}
        u_n &= \sqrt{T_0}\, \cos\left(\phi_0 + n\epsilon^2\sqrt{h^2 - T_0} + O(n^2 \epsilon^5)\right) + n\epsilon^3 \delta'_{1,n}, \\
        w_n &= \sqrt{T_0}\, \sin\left(\phi_0 + n\epsilon^2\sqrt{h^2 - T_0} + O(n^2 \epsilon^5)\right) + n\epsilon^3 \delta'_{2,n},
        \end{aligned}
        \end{equation}
        where \(\delta_{j,n}, \delta'_{j,n}\) are bounded by constants depending only on \(M, c_0, c_1\).

        \item If \(u_0 = w_0 = 0\), then \(u_n, w_n\) are \(O(n\epsilon^3)\), in the sense that there exist positive constants \(M_\delta, K\), depending only on \(M, c_0\), such that if \(N \epsilon_0^3 \leq K\), then for \(n = 0, 1, \ldots, N\),
        \[
        T_n = n\epsilon^3 \delta_n, \quad \text{with } |\delta_n| \leq M_\delta,
        \]
        and \(u_n, w_n\) are \(O(n\epsilon^3)\).
    \end{enumerate}
\end{lemma}

\begin{proof}

\noindent (i) Let $M$ and $\epsilon_0$ be as in \eqref{assumeRnSn}. We choose $M_\delta$ large enough so that the following holds 
\begin{equation}
\label{estimM}
 \epsilon h^4+4M h(1+\epsilon^2h)+\epsilon^3 2M^2\le\frac12 M_\delta.
 \end{equation}
 for all $\epsilon\le\epsilon_0$. We then choose $K$ small enough so that 
\begin{equation}
\label{estimK}
 \epsilon_0 h^2 K\le 1 
\end{equation}
and also
\begin{equation}
\label{condM}
KM_\delta<\frac12c_1,\ \ \ \text{ and }\ \  \  KM_\delta<\frac{h^2-c_0^2}2:=c_2.
\end{equation}
{\em 1. Estimate of $T_n$.}
We deduce a recurrence relation for $T_n$ from which we prove, by induction, the estimate \eqref{solrecTp}. 

\noindent Squaring both sides in equations \eqref{recurrenceuw} and adding them we obtain
\begin{multline}
\label{recurTn}
T_{n+1}=T_n+\epsilon^4T_n\left( h^2-T_n\right)+\epsilon^3\tilde{R}_n\\
\text{where }
\tilde{R}_n=-2R_n\left(u_n-\epsilon^2w_n\sqrt{h^2-T_n}\right)+2S_n\left(w_n+\epsilon^2u_n\sqrt{h^2-T_n}\right)+\epsilon^3{R}_n^2+\epsilon^3{S}_n^2
\end{multline}
We prove the estimate \eqref{solrecTp} for all $n\le N$ where $N$ satisfies \eqref{assumeN}  by induction on $n$.

Of course, $\delta_0=0$, so \eqref{solrecTp} holds for $n=0$. Next, we assume that for some $n$ with $1\le n\le N$  we have $T_k=T_0+k\epsilon^3\delta_k$ with $|\delta_k|\le M_\delta$ for all $k=0,1,\ldots n-1$ and prove this estimate for $T_n$. 
In view of \eqref{condM}, \eqref{condu0w0}, and \eqref{assumeN} we have
\begin{equation}
\label{boundsTk}
T_k\ge \frac12 c_1 {>0}\ \text{ and }\ h^2-T_k\ge c_2 {>0}\ \ \ \text{for }k=0,1,\ldots,n-1
\end{equation}
We  sum for $n$ from $0$ to $n-1$ in \eqref{recurTn} and obtain 
$$
T_n=T_0+\epsilon^4\sum_{k=0}^{n-1}T_k\left( h^2-T_k\right)+\epsilon^3\sum_{k=0}^{n-1}\tilde{R}_k.
$$
Therefore, using the fact that $T_k>0$ and $h^2-T_k>0$, it follows that
\begin{equation}
\label{TnmT0}
\left|T_n-T_0\right|\le \epsilon^4h^2\sum_{k=0}^{n-1}T_k +\epsilon^3\sum_{k=0}^{n-1}\left| \tilde{R}_k\right|.
\end{equation}
Using the fact that $|u_k|,\,|w_k|\le \sqrt{T_k}<h$ we have
\begin{equation}
\label{sumtRk}
\left| \tilde{R}_k\right|\le 4Mh (1+\epsilon^2h)+2\epsilon^3M^2 .
\end{equation}
Hence
$$\sum_{k=0}^{n-1}\left| \tilde{R}_k\right| \le 4nMh(1+\epsilon^2h) +\epsilon^3 2nM^2. $$
Also, by the induction hypothesis, 
$\sum_{k=0}^{n-1}T_k\le n\left(T_0+K\frac12 M_\delta\right)$. Using this and \eqref{sumtRk} in \eqref{TnmT0} we obtain
\begin{equation}
\label{prelimestimtnmt0}
\left|T_n-T_0\right|\le n \epsilon^3 \left[ \epsilon h^2 \left(T_0+K\frac12 M_\delta\right)+4Mh(1+\epsilon^2h)+\epsilon^3 2M^2 \right]\le n \epsilon^3\,M_\delta
\end{equation}
where the last inequality holds for $\epsilon\le\epsilon_0$ since we chose $K$ so that \eqref{estimK} holds,
and since, by  \eqref{estimM}, we have
\begin{equation}
\label{estimMraw}
 \epsilon h^2 T_0+4Mh(1+\epsilon^2h) +\epsilon^3 2M^2\le\frac12 M_\delta
 \end{equation}

Therefore, by \eqref{prelimestimtnmt0}, the estimate \eqref{solrecTp} holds for $T_n$, with the same bound $M_\delta$ for $\delta_n$ as for $\delta_1,\ldots,\delta_{n-1}$. The induction step is proved. 

\noindent{\em 2. Estimate of $u_n+iw_n$.} Denote $u_n+iw_n=\sqrt{T_n}A_n$ (of course, $A_n=e^{i\phi_n}$ for some real $\phi_n$). From \eqref{recurrenceuw} we obtain
$$\sqrt{T_{n+1}}A_{n+1}=\sqrt{T_n}A_n+i\epsilon^2A_n\sqrt{T_n}\sqrt{h^2-T_n}+ \epsilon^3(-R_n+iS_n)$$
or, dividing by $\sqrt{T_{n+1}}$ (since $T_k>0$ for all $k\le N$),
\begin{equation}
\label{recurAn}
A_{n+1}=\frac{\sqrt{T_{n}}}{\sqrt{T_{n+1}}}C_nA_n +\epsilon^3\frac{-R_n+iS_n}{\sqrt{T_{n+1}}}
\end{equation}
with $C_n= 1+i\epsilon^2 \sqrt{h^2-T_n}$.

Denote, for $n\ge 0$, $\tilde{C}_n=\frac{\sqrt{T_{n}}}{\sqrt{T_{n+1}}}C_n$ and let $B_0 := A_0$ and $A_n=B_n\prod_{\ell=0}^{n-1}\tilde{C}_\ell$ for $n \ge 1$.

The recursion \eqref{recurAn} becomes, upon changing $n$ to $n-1$,
\begin{equation}
\label{RecBn}
B_{n}=B_{n-1}+\epsilon^3 \tilde{S}_{n-1},\ \ \ \text{where } \tilde{S}_{n-1}=\frac1{\prod_{\ell=0}^{n-1}C_\ell}   \frac{-R_{n-1}+iS_{n-1}}{\sqrt{T_{0}}}. 
\end{equation}
Summing \eqref{RecBn} from $1$ to $n$ it follows that
\begin{equation}
\label{formBn}
B_n=A_0+\epsilon^3 \sum_{k=0}^{n-1} \tilde{S}_k
\end{equation}
and therefore, using the fact that $A_0=\dfrac{u_0+iw_0}{\sqrt{T_0}}$,
\begin{equation}
\label{interAn}
A_n=\frac{1}{\sqrt{T_n}}\, \prod_{\ell=0}^{n-1}C_\ell\, \left(u_0+iw_0+\epsilon^3 \sum_{k=0}^{n-1}\frac{-R_k+iS_k}{\prod_{\ell=0}^{k}C_\ell}\right).
\end{equation}

Noting that $|C_\ell|\ge 1$  the last sum is less, in absolute value, than $n\epsilon^3M\sqrt{2}$. 

For an upper estimate of $|\ln C_\ell |$ note that we have $\Re \ln(1+i\epsilon^2x)=\ln\sqrt{1+\epsilon^4x^2}=O(\epsilon^4)$ and $\Im \ln(1+i\epsilon^2x)=i\epsilon^2x+O(\epsilon^6)$ where for $x\in(0,h)$ the terms $O(\epsilon^4)$ have an absolute bound. Therefore
\begin{equation}
 \label{estimPCj}
\prod_{k=0}^{n-1}C_k=\exp\left(\sum_{k=0}^{n-1}\ln C_k\right) \\
= \exp\left( i\epsilon^2\sum_{k=0}^{n-1} \sqrt{h^2-T_k} +n\epsilon^4\eta'_n\right)
 \end{equation}
 where $\eta'_n$ has a uniform bound. 
 
Further noting that by \eqref{solrecTp} we have
\begin{equation}
 \label{estimSTn}
 \sqrt{T_n}=\sqrt{T_0+n\epsilon^3\delta_n}=\sqrt{T_0}+O\left(n\epsilon^3\right)
 \end{equation}
 Using \eqref{estimPCj} and \eqref{defphi0} in \eqref{interAn} we obtain \eqref{unwn}.
  
 To keep track of $\eta_n$, from \eqref{interAn}, \eqref{estimPCj}, \eqref{estimSTn} we have
 $$1+n\epsilon^3\eta_n:= \left(1 + \epsilon^3 \dfrac{e^{-i\phi_0}}{\sqrt{T_0}}\sum_{k=0}^{n-1}\frac{-R_k+iS_k}{\prod_{\ell=0}^{k}C_\ell} \right)e^{n\epsilon^4\eta'_n} \le \left( 1+e^{-i\phi_0}n\epsilon^3M\sqrt{2}\eta''_n\right)e^{n\epsilon^4\eta'_n} $$
 where $|\eta''_n|\le 1$ and $\eta'_n$ has a bound in terms of $T_0$ only.

 \ 

\noindent{\em 3. The case $n\epsilon^{5/2}\ll1$.} We show that $\sqrt{h^2-T_k}$ is close to $\sqrt{h^2-T_0}$: using \eqref{boundsTk} we have
\begin{equation}
 \label{estimsq}
 \left|\sqrt{h^2-T_0- k\epsilon^3\delta_k}-\sqrt{h^2-T_0}\right|\le k\epsilon^3M_\delta \,\frac{1}{2\sqrt{c_2} } 
\end{equation}
therefore
$$\epsilon^2\left| \sum_{k=0}^{n-1} \sqrt{h^2-T_k}-\sqrt{h^2-T_0}\right|\le n^2\epsilon^5M_\delta \,\frac{1}{2\sqrt{c_2} }$$
and  we obtain \eqref{unwnsmaller}. 

(iii)  If $u_0=w_0=0$ then $T_0=0$ and the arguments above in (i) for estimating $T_n$ are still valid and it follows that $T_n=O(n\epsilon^3)$ for $n = 0,1,...,N$ where $N$ satisfies $N\epsilon_0^3 \le K$. Moreover, a direct iteration of \eqref{recurrenceuw} gives that $|u_n|,|w_n|\le Cn\epsilon^3$ for a suitable $C>0$ (which is easily proven by induction).
\end{proof}

\begin{lemma}
\label{boundingremainders} Fix $h>0$. Consider the recurrence \eqref{recurrenceuw} with initial conditions $u_0,w_0$ satisfying 
\begin{equation}\label{assumpincond}
  0<u_0^2 + w_0^2 <h^2.
  \end{equation}
Let $c_0,c_1$ be positive constants so that
$$0<c_1<u_0^2 + w_0^2 <c_0^2<h^2.$$
Then there exist positive constants $\epsilon_0$, $K_0, M$ which depend only on $c_0,c_1$ so that for any $N$ satisfying $N \epsilon_0^3 \le K_0$, the remainders $R_n$ and $S_n$ in \eqref{recurrenceuw} satisfy $|R_n| \le M$, $|S_n| \le M$ for all $n = 0,1,...,N-1$ and for all $\epsilon \in \left[0,\epsilon_0 \right]$.     
\end{lemma}
\begin{proof}
Denote 
$$T_n = u_n^2 + w_n^2.$$
We first fix the constants $\epsilon_0,K_0, M$.

Let $a_1,a_2$ be such that
\begin{equation}\label{c1a2a2c0}
0<c_1< a_1<u_0^2 + w_0^2 <a_2< c_0^2<h^2.
\end{equation}
Applying Lemma\,\ref{L2} for the initial conditions $v_0=h-u_0$ and $w_0$ there exist $\epsilon_0'$ and $M$, dependent only on $c_0$, so that $|R_0|$ and $|S_0|$ are bounded by $M$, for all $\epsilon \in [0,\epsilon_0']$.

Applying Lemma \ref{Tn}(i) for $N=1$, there exist positive constants $M_{\delta}, K$ depending on $M, c_0,c_1$, namely satisfying \eqref{estimM}, \eqref{estimK}, \eqref{condM}, such that if $\epsilon_0^3 \le K$ (that is, we choose $\epsilon_0$ such that it satisfies $\epsilon_0 \le \epsilon_0'$ and $\epsilon_0^3 \le K$) then
$ T_1 = T_0 + \epsilon^3 \delta_1$
with $|\delta_1| \le M_{\delta}$ for all $\epsilon \in \left[0,\epsilon_0 \right]$.

Let 
\begin{equation}\label{defK0}
    K_0 := \min \left\{K, \dfrac{T_0 - a_1}{M_{\delta}}, \dfrac{a_2 - T_0}{M_{\delta}} \right\}.
\end{equation}

We now prove by induction that if $T_{n}$ satisfies $0<c_1< a_1\le T_{n} \le a_2< c_0^2<h^2$ and $|R_n|,|S_n|$, are bounded by $M$ for all positive $\epsilon\le\epsilon_0$ and for all $n\le N-1$ then the inequalities hold for $n=N$ (with the same $M$ and $\epsilon_0$) as long as $N\epsilon_0^3\le K_0$.

(i) The step $N = 1$ is valid by the  choice of $M, K_0,\epsilon_0$ above.

(ii) Now assume the inequalities hold for $n\le N-1$ and prove them for $n=N$, under the assumption  $N\epsilon_0^3\le K_0$. By the induction hypothesis we have  $|R_n|,|S_n|\le M$ for $n\le N-1$. By  lemma \ref{Tn},  (i)  with the  same $K, M_\delta$ as at (i) (since the quantities depend only on $c_0$), we have $T_N=T_0+N\epsilon^3\delta_N$  with $|\delta_N|\le M_\delta$ if $N\epsilon_0^3<K$  (which is clearly satisfied  since $K_0\le K$). Then
$T_N\le T_0+K_0M_\delta\le  T_0+(a_2-T_0)=a_2$
and similarly,
$T_N\ge T_0-K_0M_\delta\ge  a_1$. Then $c_1<T_N<c_0^2$ and Lemma\,\ref{L2} can be applied, yielding then $|R_N|,\,|S_N|$ are bounded by the same $M$ for $\epsilon\le\epsilon_0$.

\end{proof}

\subsection{End of the proof of Theorem\,\ref{mainth}}\label{endofproof}

\subsubsection{Proof of (i) }\label{Pfofi} This follows by a straightforward calculation, see the Appendix \ref{v0is0}. 

\subsubsection{Proof of (ii) }\label{Pfofii} This was proved in Lemma\,\ref{Tn} (iii).

\subsubsection{Proof of (iii) }\label{Pfofiii}
By Lemma\,\ref{boundingremainders},  Lemma\,\ref{Tn},  and the rescaling \eqref{rescaling} the theorem follows.

\subsubsection{Proof of (iii)}\label{Pfofiii2}  When $n\epsilon^{5/2}\ll1$, we use the rescaling \eqref{rescaling} and proof of Lemma\,\ref{Tn} (ii).

\section{Comparison between the approximate solutions \eqref{sincos} and numerical solutions}\label{numerical}
The agreement between our perturbative formulas for $v$ and $w$ and their numerical simulations is shown in figure\,\ref{finalcomp}.
We found empirically that the numerical solver is accurate with at least ten digits. The numerical values were calculated using Wolfram Mathematica with WorkingPrecision $=32$, AccuracyGoal $=30$ and MaxSteps $=10^8$. 

Figure\,\ref{error_analysis} shows that the maximal error stays small  for hundreds of thousands of oscillations.
\begin{figure}[hbt!]
    \centering
    \includegraphics[width=0.5\textwidth]{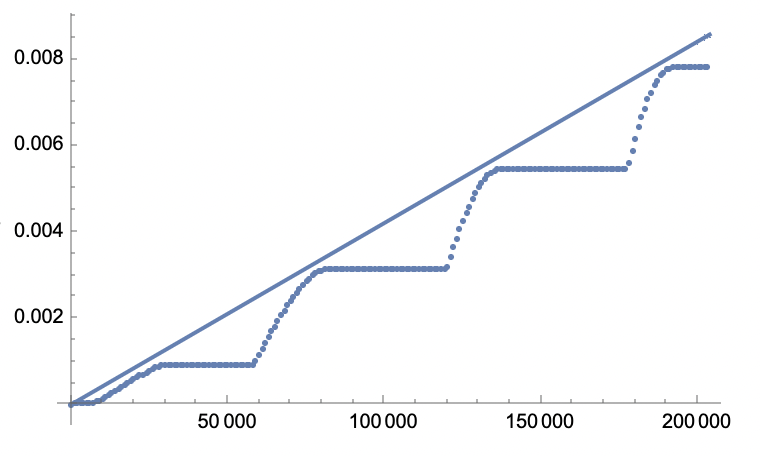}
    \caption{Error Analysis: maximum value of the difference between the numerical and theoretical value of  $v_n$ on the interval $[0,n]$ as a function of $n$.}
    \label{error_analysis}
\end{figure}

\subsection{Open questions}
The first question is whether better approximations, for longer times, can be obtained by using a higher order truncation, or perhaps by using other slow variables. The question of determining, rigorously, the region that the trajectory fills densely, as seen in figure \ref{fig:evolution} (C) remains open, as well as the question of existence of regions of order in the phase space.


\appendix
\section{Solution through perturbation expansion}
\label{epsilon square term}
\noindent The $\epsilon$ term in the perturbation expansion for $x(t)$ and $y(t)$ are 
\begin{align*}
    x_{\epsilon}(t) &=   -\dfrac{4}{3} \sin ^2\left(\frac{t}{2}\right) \Big((x_0 \dot{y}_0+y_0 \dot{x}_0) \sin{t} +(x_0 y_0-\dot{x}_0 \dot{y}_0)\cos t  +2 x_0 y_0+\dot{x}_0 \dot{y}_0\Big),  \\
    y_{\epsilon}(t) &=  -\frac{2}{3} \sin ^2\left(\frac{t}{2}\right) \Big( \left(x_0^2-y_0^2-\dot{x}_0^2+\dot{y}_0^2\right)\cos{t} + 2  (x_0 \dot{x}_0-y_0 \dot{y}_0)\sin{t} + \dot{x}_0^2-\dot{y}_0^2 +2 x_0^2-2 y_0^2 \Big). 
\end{align*}
The $\epsilon^2$ term in the perturbation expansion for $x(t)$ and $y(t)$ are $x_{\epsilon^2}(t) = f_1(t) + t f_2(t)$, $y_{\epsilon^2}(t) = g_1(t) + t g_2(t)$, where
\begin{align*}
f_1(t) &= \frac{1}{144} \big(16 \cos (2 t) \left(x_0^3+x_0 y_0^2+4 x_0 \dot{x}_0^2+4 y_0 \dot{x}_0 \dot{y}_0\right)+29 x_0^3 \cos (t)+3 x_0^3 \cos (3 t)-48 x_0^3\\
&+65 x_0^2 \dot{x}_0 \sin (t)-16 x_0^2 \dot{x}_0 \sin (2 t)+9 x_0^2 \dot{x}_0 \sin (3 t)+29 x_0 y_0^2 \cos (t)+3 x_0 y_0^2 \cos (3 t)\\
&-48 x_0 y_0^2 +86 x_0 y_0 \dot{y}_0 \sin (t) +32 x_0 y_0 \dot{y}_0 \sin (2 t)+6 x_0 y_0 \dot{y}_0 \sin (3 t)-55 x_0 \dot{x}_0^2 \cos (t)\\
&-9 x_0 \dot{x}_0^2 \cos (3 t) -189 x_0 \dot{y}_0^2 \cos (t)-3 x_0 \dot{y}_0^2 \cos (3 t)+192 x_0 \dot{y}_0^2 -21 y_0^2 \dot{x}_0 \sin (t)\\
&-48 y_0^2 \dot{x}_0 \sin (2 t) +3 y_0^2 \dot{x}_0 \sin (3 t)+134 y_0 \dot{x}_0 \dot{y}_0 \cos (t)-6 y_0 \dot{x}_0 \dot{y}_0 \cos (3 t)+32 \dot{x}_0 \dot{y}_0^2 \sin (2 t)\\
&-192 y_0 \dot{x}_0 \dot{y}_0+5 \dot{x}_0^3 \sin (t) +32 \dot{x}_0^3 \sin (2 t)-3 \dot{x}_0^3 \sin (3 t)+5 \dot{x}_0 \dot{y}_0^2 \sin (t)-3 \dot{x}_0 \dot{y}_0^2 \sin (3 t)\big), \\
f_2(t) &= \frac{1}{144} \big(60 x_0^3 \sin (t)-60 x_0^2 \dot{x}_0 \cos (t)+60 x_0 y_0^2 \sin (t)-168 x_0 y_0 \dot{y}_0 \cos (t)+168 y_0 \dot{x}_0 \dot{y}_0 \sin (t)\\
&+60 x_0 \dot{x}_0^2 \sin (t) -108 x_0 \dot{y}_0^2 \sin (t) +108 y_0^2 \dot{x}_0 \cos (t)-60 \dot{x}_0 \left(\dot{x}_0^2+\dot{y}_0^2\right) \cos (t)\big), \\
g_1(t) &= \frac{1}{144} \big(16 \cos (2 t) \left(x_0^2 y_0+4 x_0 \dot{x}_0 \dot{y}_0+y_0^3+4 y_0 \dot{y}_0^2\right)+29 x_0^2 y_0 \cos (t)+3 x_0^2 y_0 \cos (3 t)-48 x_0^2 y_0 \\
&-21 x_0^2 \dot{y}_0 \sin (t)-48 x_0^2 \dot{y}_0 \sin (2 t)+3 x_0^2 \dot{y}_0 \sin (3 t)+86 x_0 y_0 \dot{x}_0 \sin (t)+32 \dot{y}_0^3 \sin (2 t)\\
&+32 x_0 y_0 \dot{x}_0 \sin (2 t) +6 x_0 y_0 \dot{x}_0 \sin (3 t) +134 x_0 \dot{x}_0 \dot{y}_0 \cos (t)-6 x_0 \dot{x}_0 \dot{y}_0 \cos (3 t)\\
&-192 x_0 \dot{x}_0 \dot{y}_0+29 y_0^3 \cos (t) +3 y_0^3 \cos (3 t)-48 y_0^3+65 y_0^2 \dot{y}_0 \sin (t)-16 y_0^2 \dot{y}_0 \sin (2 t)\\
&+9 y_0^2 \dot{y}_0 \sin (3 t)-y_0 \left(189 \dot{x}_0^2+55 \dot{y}_0^2\right) \cos (t)-3 y_0 \dot{x}_0^2 \cos (3 t)+192 y_0 \dot{x}_0^2-3 \dot{y}_0^3 \sin (3 t)\\
&-9 y_0 \dot{y}_0^2 \cos (3 t)+5 \dot{x}_0^2 \dot{y}_0 \sin (t) +32 \dot{x}_0^2 \dot{y}_0 \sin (2 t)-3 \dot{x}_0^2 \dot{y}_0 \sin (3 t)+5 \dot{y}_0^3 \sin (t)\big), \text{ and }\\
g_2(t) &= \frac{1}{144} \big(60 x_0^2 y_0 \sin (t)+108 \dot{x}_0^2 \dot{y}_0 \cos (t)-168 x_0 y_0 \dot{x}_0 \cos (t)+168 x_0 \dot{x}_0 \dot{y}_0 \sin (t) \\
&+60 y_0^3 \sin (t) -60 y_0^2 \dot{y}_0 \cos (t)-108 y_0 \dot{x}_0^2 \sin (t)+60 y_0 \dot{y}_0^2 \sin (t)-60 \dot{y}_0 \left(\dot{x}_0^2+\dot{y}_0^2\right) \cos (t)\big).
\end{align*}

\section{Higher order expansions for the slow variables}
\label{vwExpansions}

\noindent Let $v(y) = v^{[0]}(y)+\epsilon v^{[1]}(y) + \epsilon^2 v^{[2]}(y) + \epsilon^3 R(y)$ and $w(y) = w^{[0]}(y)+\epsilon w^{[1]}(y) + \epsilon^2 w^{[2]}(y) + \epsilon^3 S(y)$  (where the remainders $R,S$ also depend on $\epsilon$). Clearly, $v^{[0]}(y) = v_0, w^{[0]}(y) = w_0$. To calculate the next two terms, we expand $2\epsilon (-x^2+y^2)$:
\begin{equation*}
    2\epsilon (-x^2+y^2) = \epsilon T_1(y)+\epsilon^2 T_2(y)+O(\epsilon^3)
\end{equation*}
where 
\begin{equation}
\label{T1}
    T_1(y) = -\frac{2 r_0}{v_0}+\frac{4 y^2 \left(h v_0-w_0^2\right)}{v_0^2}-\frac{4 r_0 w_0 y \sqrt{v_0-y^2}}{v_0^2}
\end{equation}
where we used the notation
$$r_0 := \sqrt{2 h v_0-v_0^2-w_0^2}$$ and
\begin{equation}
    \label{T2}
    T_2(y)=v^{[1]}(y)A(y)+w^{[1]}(y)B(y)+C(y)
\end{equation}
where the exact expressions of $A(y),B(y), C(y)$ are found in the Appendix \ref{ABC}.

Similarly,
\begin{equation*}
    -\epsilon\left[ \frac{\left(x^2-y^2\right)\dot{x}}{\sqrt{v-y^2}}+2xy \right] \equiv \epsilon \tilde{T}_1(y)+\epsilon^2 \tilde{T}_2(y)+O(\epsilon^3)
\end{equation*}
where 
\begin{equation}\label{T1Tilde}
    \tilde{T}_1(y) = -\frac{2 w_0 y^2 \left(-3 h v_0+2 v_0^2+2 w_0^2\right)}{v_0^3}-\frac{r_0 y \left(y^2 \left(r_0^2-v_0^2-3 w_0^2\right)+v_0 \left(-r_0^2+2 v_0^2+2
   w_0^2\right)\right)}{v_0^3 \sqrt{v_0-y^2}}-\frac{r_0^2 w_0}{v_0^2} 
\end{equation}
and 
\begin{equation}
\label{eqtT2}
 \tilde{T}_2(y)=v^{[1]}(y) \tilde{A}(y)+w^{[1]}(y) \tilde{B}(y)+ \tilde{C}(y)
  \end{equation}
where the exact expressions of $\tilde{A}(y),\tilde{B}(y), \tilde{C}(y)$ are found in the Appendix \ref{ABC}.

Using \eqref{T1} in $v^{[1]}(y)=\int_0^y T_1(\xi)d\xi$ we obtain
\begin{equation}\label{v1}
         v^{[1]}(y) = \frac{-2}{3 v_0^2} \Big(-2 w_0 \left(v_0-y^2\right){}^{3/2} r_0+2 v_0^{3/2} w_0 r_0-2 h v_0 y^3+ 6 h v_0^2 y -3 v_0 w_0^2 y-3 v_0^3 y+2 w_0^2 y^3\Big). 
\end{equation}
Using \eqref{T1Tilde} in $w^{[1]}(y)=\int_0^y \tilde{T}_1(\xi) d\xi$, we get
\begin{equation}\label{w1}
    \begin{split}
       w^{[1]}(y) =& \frac{1}{3v_0^3} \Big(6 h v_0 w_0 y^3-6 h v_0^2 w_0 y-4 v_0^2 w_0 y^3-5 v_0^{7/2} r_0 +3 v_0^3 w_0 y -v_0^{3/2} w_0^2r_0+2 h v_0^{5/2} r_0+3 v_0 w_0^3 y \\
       &-4 w_0^3 y^3 + \sqrt{v_0-y^2}(-2 v_0^2 y^2 r_0 -2 h v_0^2  r_0+5 v_0^3  r_0+v_0 w_0^2 r_0 +2 h v_0 y^2  r_0-4 w_0^2 y^2 r_0 )\Big). 
    \end{split}
\end{equation}
In the above, we find the integrals using the fact that each term of $T_1(\xi)$ and $\tilde{T}_1(\xi)$ has an analytic antiderivative and therefore we substitute the integration limits $y$ and $0$ in the antiderivative function. 

Substituting \eqref{v1}, \eqref{w1} in \eqref{T1} and integrating from $0$ to $y$ along $\mathcal{C}_{v_0}$, we obtain:
\begin{equation}
\label{valv2}
v^{[2]}(y)= \dfrac{7 w_0 r_0}{3}\int_0^y\frac{ds}{\sqrt{v_0-s^2}}
 +\text{Polynomial}(y, \sqrt {v_0-{y}^2}).
 \end{equation}
 
Substituting \eqref{v1}, \eqref{w1} in \eqref{T1Tilde} and integrating from $0$ to $y$ along $\mathcal{C}_{v_0}$, we obtain:
\begin{equation}
\label{valw2}
w^{[2]}(y)=\dfrac{7 (h-v_0) r_0}{3}\int_0^y\frac{ds}{\sqrt{v_0-s^2}}
 + \dfrac{\text{Polynomial}(y, \sqrt {v_0-{y}^2})}{\sqrt {v_0-{y}^2}}. 
 \end{equation}
After one integration along $\mathcal{C}_{v_0}$, the value of $v(y)$ and $w(y)$ at $y=0$, respectively, become
\begin{equation*}
    v_1:=\lim_{y\to 0, y\in \mathcal{C}_{v_0}}v(y), \  w_1:=\lim_{y\to 0, y\in \mathcal{C}_{v_0}}w(y).
\end{equation*}
Using \eqref{vwhigherExpansion}, \eqref{v1}, \eqref{w1}, \eqref{valv2}, \eqref{valw2}, we obtain
\begin{equation*}
    v_1 = v_0 + \frac{14 \pi}{3} \epsilon^2 w_0 r_0 + \epsilon^3 R_0, \  w_1 = w_0 + \frac{14 \pi}{3} \epsilon^2 (h-v_0)  + \epsilon^3 S_0.
\end{equation*}
We used that after analytic continuation along one loop on $\mathcal{C}_{v_0}$, $\arcsin{\left( \frac{y_0}{\sqrt{v_0}} \right)}$ becomes $\arcsin{\left( \frac{y_0}{\sqrt{v_0}} \right)} + 2\pi$, while the other terms are not ramified and return to their zero value. 

\section{Exact formulas for series coefficients}\label{ABC}
Denoting $r_0 = \sqrt{2 h v_0-v_0^2-w_0^2}$, the exact formulas for $A(y),B(y),C(y)$ are given by:
\begin{align*}
     A(y) &= 2+\frac{8 w_0^2 y^2}{v_0^3}-\frac{2 \left(2 h
   y^2+w_0^2\right)}{v_0^2}-\frac{2 w_0 y^3 \left(6 h v_0-2 v_0^2-4 w_0^2\right)+2 w_0 y \left(-4 h v_0^2+3 v_0 w_0^2+v_0^3\right)}{r_0 v_0^3 \sqrt{v_0-y^2}}\\
   B(y) &= \frac{4 w_0}{v_0} -\frac{8 w_0 y^2}{v_0^2} +\frac{4 y \sqrt{v_0-y^2} \left(-2 h v_0+v_0^2+2 w_0^2\right)}{r_0 v_0^2} \\
    C(y) &= \frac{4 r_0^2 y}{v_0}-\frac{4 y^3 \left(12 h v_0-5 v_0^2-15 w_0^2\right)}{3 v_0^2} +\frac{8 y^5 \left(3 h v_0-v_0^2-6
   w_0^2\right)}{3 v_0^3} \\
   &+\sqrt{v_0-y^2} \left(\frac{12 r_0 w_0 y^2}{v_0^2}-\frac{8 y^4 \left(9 h v_0 w_0-4 v_0^2 w_0-6 w_0^3\right)}{3 r_0 v_0^3}\right)
\end{align*}
The exact expressions of $\tilde{A}(y),\tilde{B}(y), \tilde{C}(y)$ are:
\begin{align*}
    \tilde{A}(y) &= -\frac{2
   w_0^3}{v_0^3}+\frac{2 h w_0}{v_0^2}-\frac{4 w_0 y^2 \left(3 h v_0-v_0^2-3 w_0^2\right)}{v_0^4} \\
   &- \frac{1}{2 r_0 v_0^4
   \left(v_0-y^2\right)^{3/2}} \Big(y (r_0^2 (v_0^2 \left(-10 h y^2-15 w_0^2+y^4\right)+2 v_0^3 \left(2 h+y^2\right)+40 v_0 w_0^2 y^2-v_0^4\\
   &-18 w_0^2 y^4) -v_0^2 w_0^2 y^2 \left(10 h+3 y^2\right)-2 v_0^4 \left(3 h y^2+2 w_0^2+y^4\right)+2 v_0^3 \left(w_0^2 \left(2 h+5
   y^2\right)+h y^4\right)\\
   &+v_0^5 \left(4 h+6 y^2\right) +3 r_0^4 y^4-4 v_0^6 +3 w_0^4 y^4) \Big)\\
    \tilde{B}(y) &= \frac{6 h v_0 y^2-2 h v_0^2+3 v_0 w_0^2-4 v_0^2 y^2+v_0^3-12 w_0^2 y^2}{v_0^3} \\
    &-\frac{w_0 y \left(r_0^2 \left(7 v_0-9 y^2\right)-2
   v_0 w_0^2+v_0^2 y^2-2 v_0^3+3 w_0^2 y^2\right)}{r_0 v_0^3 \sqrt{v_0-y^2}} \\
   \tilde{C}(y)&= \frac{1}{3 v_0^4}\Big(2 w_0 (-18 v_0^2 y^4-36
   w_0^2 y^4+18 (r_0^2+v_0^2+w_0^2) y^4+26 v_0^3 y^2-42 h v_0^2 y^2+30 v_0 w_0^2 y^2 \\
   & -3 v_0^4+6 h v_0^3-3 v_0^2
   w_0^2) y\Big) +\frac{1}{3 r_0 v_0^6(v_0-y^2)^{3/2}}\Big((6 v_0^9-12 h v_0^8+21 r_0^2 v_0^7+12 w_0^2 v_0^7  -18 h r_0^2 v_0^6 \\
   &-12 h w_0^2 v_0^6+6 w_0^4 v_0^5+33 r_0^2 w_0^2
   v_0^5+y^2 (-17 v_0^8+30 h v_0^7-54 r_0^2 v_0^6-32 w_0^2 v_0^6 +54 h r_0^2 v_0^5 \\
   &+42 h w_0^2 v_0^5 -15 w_0^4 v_0^4-138 r_0^2
   w_0^2 v_0^4)+y^6 (3 r_0^6+3 v_0^2 r_0^4+9 w_0^2 r_0^4-12 v_0^4 r_0^2+9 w_0^4 r_0^2 +24 h v_0^3 r_0^2 \\
   &-12 h^2 v_0^2r_0^2 -66 v_0^2 w_0^2 r_0^2 -4 v_0^6+3 w_0^6+6 h v_0^5+3 v_0^2 w_0^4-12 v_0^4 w_0^2+24 h v_0^3 w_0^2-12 h^2 v_0^2 w_0^2) \\
   &+y^4(15 v_0^7 -24 h v_0^6 +54 r_0^2 v_0^5+41 w_0^2 v_0^5-90 h r_0^2 v_0^4-84 h w_0^2 v_0^4+9 r_0^4 v_0^3+27 w_0^4 v_0^3+36 h^2
   r_0^2 v_0^3 \\
   &+36 h^2 w_0^2 v_0^3 +204 r_0^2 w_0^2 v_0^3 -18 h r_0^4 v_0^2-36 h w_0^4 v_0^2-54 h r_0^2 w_0^2 v_0^2+9 w_0^6 v_0+18
   r_0^2 w_0^4 v_0 \\
   &+9 r_0^4 w_0^2 v_0) +y^3 (-6 v_0^{3/2} w_0 r_0^5+6 v_0^{3/2} w_0^3 r_0^3-6 v_0^{3/2} w_0
   \left(w_0^2-r_0^2\right) r_0^3)) y^2\Big)
\end{align*}

\section{The case $y_0=\dot{y}_0=0$}\label{v0is0}
In this case $x(t)=\dot{x}_0\sin t+x_0\cos t+O(\epsilon)$ and $y(t)=O(\epsilon)$ (for not large $t$), so after a translation in $t$ (of order $1$)  we can assume $\dot{x}_0=0$ (and therefore $h=x_0^2/2$).
 A direct calculation yields 
 $$x(t)=\sqrt{2}\, \sqrt{h}\, \cos \! \left(t \right)+\frac{\epsilon^{2} \left(3 \cos^{3}\left(t \right)+8 \cos^{2}\left(t \right)+15 \sin \! \left(t \right) t +5 \cos \! \left(t \right)-16\right) \sqrt{2}\, h^{\frac{3}{2}}}{18}+O(\epsilon^{4} )$$
 and
\begin{multline*}
y(t)=\frac{2 \epsilon  h \left(-2+\cos^{2}\left(t \right)+\cos \! \left(t \right)\right)}{3}\\
+\frac{\epsilon^{3} \left(\left(150 \cos \! \left(t \right)+75\right) \sin \! \left(t \right) t +2 \cos^{4}\left(t \right)+15 \cos^{3}\left(t \right)+318 \cos^{2}\left(t \right)-239 \cos \! \left(t \right)-96\right) h^{2}}{135}
+
O(\epsilon^{4} )
\end{multline*}
 where we see that $x(t)$ and $y(t)$ have the secular term $t\epsilon^2$ and $t \epsilon^3$, respectively. However, the secular terms in $v$, respectively $w$, appear as $t\epsilon^4$, respectively $t\epsilon^3$:
$$v(t)=-4\,\frac{\left(3 \cos^{4}\left(t \right)+2 \cos^{3}\left(t \right)-5\right) h^{2}}{9} \epsilon^{2}\\
+O \left(\epsilon^{4}\right) $$
and
\begin{multline*}
w(t)=-\frac{2 \left(\cos^{3}\left(t \right)-1\right) \sqrt{2}\, h^{\frac{3}{2}}}{3} \epsilon +\frac{1}{45} (-68-27 \cos^{5}\left(t \right)-60 \cos^{4}\left(t \right)-5 \cos^{3}\left(t \right)
\\+\left(-75 \sin \! \left(t \right) t +80\right) \cos^{2}\left(t \right) +80 \cos \! \left(t \right)) \sqrt{2}\, h^{\frac{5}{2}} \epsilon^{3}+O \left(\epsilon^{5}\right)
\end{multline*}

\end{document}